\documentclass{elsarticle}

\usepackage[utf8]{inputenc}
\usepackage{amsmath}
\usepackage{amssymb}
\usepackage{amsfonts}
\usepackage{amsthm}
\usepackage{geometry}
\usepackage{hyperref}
\usepackage{natbib}
\usepackage{subcaption}
\usepackage{booktabs}
\usepackage{color}
\usepackage{graphicx}
\usepackage{multirow}
\usepackage{booktabs}

\newtheorem{theorem}{Theorem}
\newtheorem{lemma}{Lemma}
\newtheorem{remark}{Remark}
\newtheorem{definition}{Definition}

\begin{document}

\begin{frontmatter}
    \title{Sharp inf-sup estimate for the Stokes equation in tight domains with periodic pillars and some numerical implications}
    \author[inst1,inst2]{Qi Xin}
    \ead{qixin1@link.cuhk.edu.cn}

    \author[inst1,inst2,inst3]{Shihua Gong\corref{cor1}}
    \ead{gongshihua@cuhk.edu.cn}

    \author[inst4]{Jinchao Xu}
    \ead{jinchao.xu@kaust.edu.sa}

    \cortext[cor1]{Corresponding author}

    \affiliation[inst1]{organization={School of Science and Engineering, The Chinese University of Hong Kong},
        addressline={Shenzhen},
        postcode={518172},
        state={Guangdong},
        country={China}}

    \affiliation[inst2]{organization={Shenzhen International Center for Industrial and Applied Mathematics, Shenzhen Research Institute of Big Data},
        addressline={Shenzhen},
        postcode={518000},
        state={Guangdong},
        country={China}}

        \affiliation[inst3]{organization={Shenzhen Loop Area  Institute },
        addressline={Shenzhen},
        postcode={518000},
        state={Guangdong},
        country={China}}


    \affiliation[inst4]{organization={Applied Mathematics and Computational Sciences, CEMSE Division, King Abdullah University of Science and Technology},
        city={Thuwal},
        postcode={23955},
        country={Saudi Arabia}}


    \begin{abstract}
        The predictive simulation of fluid dynamics in densely packed microfluidic devices, such as Deterministic Lateral Displacement (DLD) arrays, stagnates with standard iterative solvers. We show that this failure is not algorithmic but rooted in the pre-asymptotic degradation of the pressure-velocity coupling stability. For periodic pillar geometries in a generalized lattice framework, we prove that the continuous Ladyzhenskaya-Babuška-Brezzi (LBB) condition, also called the inf-sup constant, deteriorates exactly as $m^{-1}$ up to a positive multiplicative constant, where $m$ is the pillar density (the number of pillars per unit length). This induces a priori error amplification proportional to $m$ and a pressure Schur complement condition number scaling as $\mathcal{O}(m^2)$. To overcome this theoretical limit, we propose a parameter-free, adaptively scaled Augmented Lagrangian (AL) stabilization strategy with penalty $\gamma \propto m^2$. Numerical experiments on both standard square and asymmetric DLD arrays validate the theoretical bounds: the AL method reduces outer FGMRES iterations from 437 to 22 on a 1.85M-DoF square array and from 687 to 24 on a 1.77M-DoF DLD array.
    \end{abstract}
\end{frontmatter}


\section{Introduction}\label{sec:intro}

Microfluidic devices featuring dense periodic pillar arrays have advanced biomedical engineering, chemical processing, and precise particle manipulation. A prominent exemplar is the Deterministic Lateral Displacement (DLD) device \cite{huang2004continuous, McGrath2014, Inglis2006, Dincau2018}, widely employed in high-throughput cell sorting and biophysical fractionations. These devices typically consist of thousands to millions of microscopic solid pillars arranged in a shifted Bravais lattice to control fluid streamlines and particle trajectories.

While the fabrication of such devices has matured, the predictive computational modeling of the fluid dynamics and fluid-structure interactions (FSI) within these dense arrays remains open. As the pillar density increases ( quantified by the number $m\to \infty$ of pillars in a unit length , or the characteristic cell size $\epsilon \to 0$), standard numerical solvers for the incompressible Stokes equations stagnate \cite{Meier2022}.

The performance of standard iterative frameworks and classical block-preconditioning techniques (e.g., standard Schur complement approximations) degrades in dense pillar arrays. The root cause lies in the mathematical structure of the saddle-point problem. The convergence rate of preconditioned iterative solvers is dictated by the stability of the pressure-velocity coupling, quantified by the continuous and discrete inf-sup (LBB) constants \cite{boffi2013mixed, Mardal2010, loghin2004analysis}. The inf-sup constant depends on the underlying domain geometry, deteriorating in elongated, anisotropic, or heavily perforated configurations \cite{Chizhonkov2000, Dobrowolski2003, sande2025robust}. As the pillar packing density increases, the stability of the system collapses, rendering standard parameter tuning ineffective.

The magnitude of the inf-sup constant $\beta$, see definition in \eqref{eq:infsup_def}, dictates both the continuous physical stability and the numerical resolvability of the fluid system:
\begin{itemize}
    \item {A Priori Error Amplification:} At the level of numerical discretization, Brezzi's theorem \cite{Brezzi1974, boffi2013mixed, Girault1986, Ern2004} establishes that the a priori error estimate for any stable mixed finite element method is amplified by a factor proportional to $\beta_h^{-1}$, where $\beta_h$ is the discrete inf-sup constant. If the continuous constant $\beta$ approaches zero, the discrete inf-sup constant $\beta_h$ inherits this geometric degradation \cite{boffi2013mixed}, guaranteeing a loss of accuracy. As the geometric penalty drives $\beta_h$ toward zero, this pre-factor amplifies the discretization error. The optimal asymptotic rates are overwhelmed unless prohibitively fine meshes offset this algebraic amplification.
    \item {Stagnation of Iterative Solvers:} At the linear algebra level, the condition number of the pressure Schur complement operator $S$ is proportional to $\beta_h^{-2}$. This ill-conditioning dictates the stagnation of classical segregated iterative methods (such as Uzawa or block preconditioner algorithms) \cite{Mardal2010,Bacuta2006} and neutralizes the efficiency of standard block-diagonal preconditioners in Krylov subspace methods (e.g., GMRES).
\end{itemize}

While the degradation of the continuous inf-sup constant in slender domains is documented in the literature \cite{Chizhonkov2000,Dobrowolski2003, sande2025robust}, its precise pre-asymptotic estimate within the multiply-connected topologies of dense pillar arrays lacks explicit quantitative characterization. Classical homogenization theory \cite{Allaire1991,Lu2020} studies the asymptotic limit as the characteristic cell size $\epsilon \to 0$ (or equivalently, the pillar density $m \to \infty$), smearing out the discrete solids to yield a continuous macroscopic Darcy's law. However, practical CFD simulations of DLD devices operate in the pre-asymptotic regime: the pillars are discrete and finite, and the full Stokes equations must be solved explicitly to resolve local shear stresses for precise particle tracking.

We establish the sharp pre-asymptotic degradation rate of the continuous LBB constant $\beta$ in these generalized periodic pillar geometries. By constructing global linear pressure fields and leveraging local restriction operators, we prove that the inf-sup constant degrades exactly as $m^{-1}$ up to a positive multiplicative constant. This characterization shows that the failure of classical decoupled solvers in densely packed micro-channels is an intrinsic property of the PDE, not an implementation artifact, and exposes the a priori error amplification that standard finite element discretizations inherit.

Our sharp estimate also provides a quantitative foundation for stabilization. Advanced monolithic architectures—such as Augmented Lagrangian (AL) methods—bypass this geometric penalty entirely \cite{Benzi2006, Farrell2019, Fortin1983, Glowinski1989} when the penalty parameter is scaled proportional to $m^2$, as dictated by our pre-asymptotic bounds.

Our contributions are:
\begin{enumerate}
    \item \textbf{Sharp bound.} We prove $\beta(\Omega_m) = \Theta(m^{-1})$ for the continuous LBB constant on generalized periodic pillar arrays (Theorem~\ref{thm:main_scaling}), the first explicit pre-asymptotic characterization for multiply-connected microfluidic geometries.
    \item \textbf{Discretization consequence.} We map the sharp bound to its a priori error implication: the standard Brezzi estimate is amplified by $\beta_h^{-1} = \mathcal{O}(m)$, and we numerically expose an asymmetric pollution under macroscopic boundary conditions. As the mesh size $h$ is refined such that $mh\sim \mathcal{O}(1) $, under Dirichlet inflow profiles the relative pressure error stagnates at $\mathcal{O}(1)$, whereas under Neumann pressure-drop conditions first-order convergence is retained.
    \item \textbf{Algorithmic consequence.} We show that the pressure Schur complement satisfies $\kappa(M_p^{-1}S) \ge \mathcal{O}(m^2)$ and derive that setting the Augmented Lagrangian penalty as $\gamma \propto m^2$ restores $\mathcal{O}(1)$ outer iteration counts. The resulting parameter-free scaling reduces outer FGMRES iterations from 437 to 22 on a 1.85M-DoF square array and from 687 to 24 on a 1.77M-DoF DLD array.
\end{enumerate}

Section~\ref{sec:problem_main_results} states the geometric setting and the main theorem. Section~\ref{sec:proofs} establishes three lemmas. Section~\ref{sec:proofs_main} proves the main theorem. Section~\ref{sec:numerical} validates the $\Theta(m^{-1})$ rate, exposes its impact on discretization accuracy, and demonstrates density-independent AL convergence. Section~\ref{sec:conclusion} concludes.


\section{Mathematical Model and Main Results}\label{sec:problem_main_results}

\subsection{Geometric Configuration of tight domains with Periodic Pillars}\label{sec:problem_main_results1}

Let the global bounding fluid domain, representing the macroscopic scale of the system, be $\Omega \subset \mathbb{R}^d$ ($d \in \{2, 3\}$), defined by a rectangular base cross-section $\Omega_{xy} = (0, L_x) \times (0, L_y)$. For $d=2$, the domain is strictly planar ($\Omega = \Omega_{xy}$), while for $d=3$, $\Omega = \Omega_{xy} \times (0, L_z)$ represents a microfluidic channel of height $L_z$. We assume that the domain is dimensionless, which means the lengths of $L_x$ and $L_y$ are of magnitude $\mathcal{O}(1)$.

To model the microscopic geometry of generic microfluidic pillar arrays, including DLD devices, we introduce a characteristic pore-scale periodicity parameter $\epsilon > 0$ and define a dimensionless reference lattice in the $x$-$y$ plane spanned by basis vectors $\hat{\mathbf{e}}_1 = (1, \delta)$ and $\hat{\mathbf{e}}_2 = (0, 1)$, where $\delta \in [0, 1)$ is the row shift fraction.

The physical lattice nodes are distributed in the cross-sectional plane according to:
\begin{equation}
    \mathbf{x}_{i,j} = \epsilon (i \hat{\mathbf{e}}_1 + j \hat{\mathbf{e}}_2), \quad (i,j) \in \mathbb{Z}^2.
\end{equation}

Let $P_0 \subset \mathbb{R}^2$ be the open, dimensionless reference cell in the cross-sectional plane centered at the origin, spanned by $\hat{\mathbf{e}}_1$ and $\hat{\mathbf{e}}_2$. We define a dimensionless reference pillar cross-section $T \subset \mathbb{R}^2$ as an arbitrary connected, closed, and bounded set with a Lipschitz boundary. To ensure that the pore-scale pillar cross-section is non-degenerate and guarantees a uniform local fluid neighborhood for subsequent theoretical restrictions, we assume there exist two concentric open balls $B(0, r_0)$ and $B(0, R_0)$ with geometric constants $0 < r_0 \le R_0$ such that:
\begin{equation}
    B(0, r_0) \subset T \subset B(0, R_0) \subset P_0.
\end{equation}

Let $a_{\epsilon} > 0$ be a dimensional scaling parameter representing the characteristic size of the physical pillars. To prevent the pillars from overlapping the cell boundaries and to strictly bound the maximum solid volume fraction, we assume there exists a positive geometric constant $\kappa \in (0, 1)$ such that the relative pillar size uniformly satisfies $\sup_\epsilon (a_{\epsilon}/\epsilon) \le \kappa$. The individual solid pillar at cell $(i,j)$ is generated by scaling the reference pillar cross-section and translating it to the lattice node. Depending on the dimension $d$, it is either a 2D inclusion in the plane or a cylinder of cross-section $a_\epsilon T$ spanning the full channel height:
\begin{equation}
    T_{\epsilon}^{i,j} =
    \begin{cases}
        \mathbf{x}_{i,j} + a_{\epsilon} T                   & \text{for } d = 2, \\
        (\mathbf{x}_{i,j} + a_{\epsilon} T) \times [0, L_z] & \text{for } d = 3.
    \end{cases}
\end{equation}

To rigorously avoid the truncation of pillars at the global lateral boundaries $\partial \Omega_{xy}$, we restrict the placement of solid pillars to the interior of the channel. We define the active index set $\mathcal{I}_{\epsilon}$ consisting of indices for which the corresponding projected physical cell $P_\epsilon^{i,j} = \mathbf{x}_{i,j} + \epsilon P_0$ is strictly contained within $\Omega_{xy}$:
\begin{equation}
    \mathcal{I}_{\epsilon} := \{ (i,j) \in \mathbb{Z}^2 : \overline{P}_{i,j} \subset \Omega_{xy} \}.
\end{equation}

Finally, the effective fluid domain, denoted by $\Omega_m$ (with $m \propto \epsilon^{-1}$ representing the pillar density in a unit length), is defined as the global bounding channel excluding the union of all fully contained solid pillars:
\begin{equation}
    \Omega_m := \Omega \setminus \bigcup_{(i,j) \in \mathcal{I}_{\epsilon}} T_{\epsilon}^{i,j}.
\end{equation}
This geometric arrangement is illustrated schematically in Figure \ref{fig:dld_geometry}, which shows the periodic parallelogram cells and the relative positions of the pillars within the array.

\begin{figure}[htbp]
    \centering
    \includegraphics[width=0.5\textwidth]{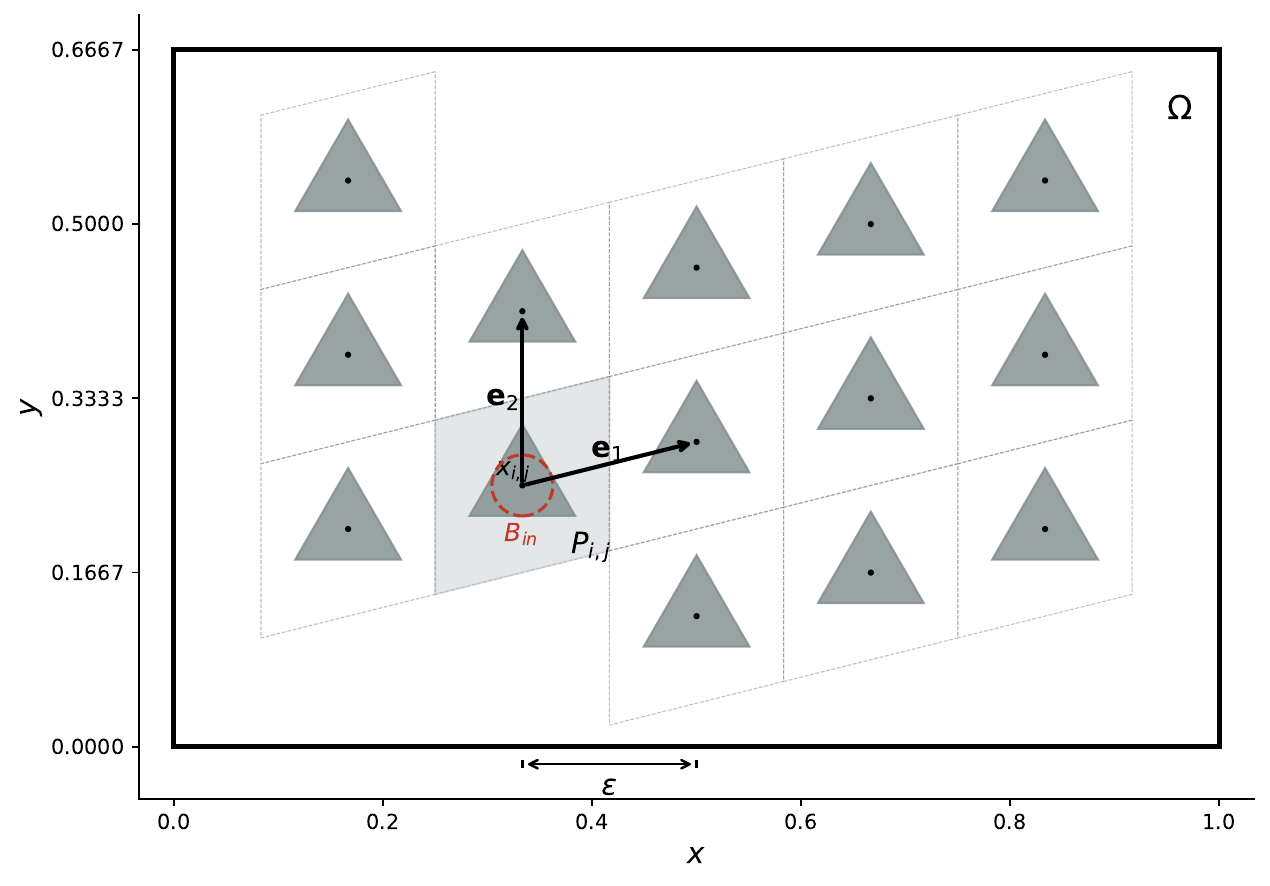}
    \caption{Geometric configuration of the periodic pillar array.}
    \label{fig:dld_geometry}
\end{figure}

\subsection{The Stokes Problem and the inf-sup Condition}

We consider the steady, incompressible Stokes equations modeling the creeping flow of a viscous fluid through the effective fluid domain $\Omega_m$:
\begin{equation} \label{eq:stokes_strong}
    \begin{cases}
        -\mu \Delta \mathbf{u} + \nabla p = f & \text{in } \Omega_m,          \\
        \nabla \cdot \mathbf{u} = 0           & \text{in } \Omega_m,          \\
        \mathbf{u} = 0                        & \text{on } \partial \Omega_m,
    \end{cases}
\end{equation}
where $\mathbf{u}$ is the fluid velocity, $p$ is the pressure, $\mu > 0$ is the dynamic viscosity, and $f$ represents the body force. The boundary $\partial \Omega_m = \partial \Omega \cup (\cup_{i,j} \partial T_{\epsilon}^{i,j})$ incorporates the no-slip condition on both the exterior channel walls and the surfaces of all internal pillars.

To cast \eqref{eq:stokes_strong} into a variational form, we introduce the standard functional spaces. The velocity space is $V(\Omega_m) = H_0^1(\Omega_m)^d$, equipped with the norm $\|\mathbf{u}\|_V = \|\nabla \mathbf{u}\|_{L^2(\Omega_m)}$. The pressure space is $Q(\Omega_m) = L_0^2(\Omega_m) = \{ q \in L^2(\Omega_m) : \int_{\Omega_m} q \, \mathrm{d}x = 0 \}$, equipped with the standard $L^2$-norm.

The well-posedness of the Stokes saddle-point problem, and the robust coupling between the velocity and pressure spaces, are governed by the classical Ladyzhenskaya-Babuška-Brezzi (LBB) condition. The continuous LBB constant for the perforated pillar-array domain $\Omega_m$ is defined as:
\begin{equation} \label{eq:infsup_def}
    \beta(\Omega_m) := \inf_{q \in Q(\Omega_m) \setminus \{0\}} \sup_{\mathbf{v} \in V(\Omega_m) \setminus \{0\}} \frac{\int_{\Omega_m} q \nabla \cdot \mathbf{v} \, \mathrm{d}x}{\|q\|_{L^2(\Omega_m)} \|\nabla \mathbf{v}\|_{L^2(\Omega_m)}}.
\end{equation}

\subsection{Main Results}

The primary objective of this paper is to establish the sharp pre-asymptotic behavior of the continuous LBB constant $\beta(\Omega_m)$ as the pillar density increases (i.e., $\epsilon \to 0$ or equivalently $m \to \infty$). We characterize this severe degradation through the following theorem, the proofs of which are deferred to Section \ref{sec:proofs_main}. Here and throughout the paper, the notation $\Theta(\cdot)$ denotes the two-sided asymptotic bound: for positive functions $f(x), g(x)$, we write $f(x) = \Theta(g(x))$ if there exist constants $c, C > 0$ and $x_0$ such that $c \, g(x) \le f(x) \le C \, g(x)$ for all $x \ge x_0$ (or as $x \to 0$), i.e., $f$ grows or decays exactly at the rate of $g$ up to positive multiplicative constants.

\begin{theorem}[Sharp Pre-asymptotic Estimate of the Continuous LBB Constant]\label{thm:main_scaling}
    Under the geometric assumptions defined in Section \ref{sec:problem_main_results1}, assume the obstacle size scales proportionally with the periodicity (i.e., $a_{\epsilon} \propto \epsilon$). Then, the continuous inf-sup constant $\beta(\Omega_m)$ satisfies the following estimate:
    \begin{equation}
        \beta(\Omega_m) = \Theta(\epsilon) = \Theta(m^{-1}).
    \end{equation}
\end{theorem}

\section{Auxiliary Lemmas for the Stability Analysis}\label{sec:proofs}

Before proving the main theorem, we establish three useful lemmas. From this section, we refer the term pillar as obstacle. Lemma \ref{lemma:poincare} provides a global Poincaré inequality that explicitly captures the geometric penalty $\sigma_{\epsilon}$ induced by the obstacle array. Lemma \ref{lemma:positivity} ensures the existence of a test pressure field with a uniformly bounded gradient-to-norm ratio, independent of the obstacle density. Finally, Lemma \ref{lemma:restriction_operator} constructs a restriction operator that enforces the microscopic no-slip conditions on global velocity fields while preserving their macroscopic divergence, providing the crucial stability estimates needed for the inf-sup lower bound. We give proof for 2-dimensional case, and the 3-dimensional case can be proved similarly.

\begin{definition}[Asymptotic Parameter for Obstacle Arrays]\label{def:asymptotic_parameter}
    For the geometric setting in Section \ref{sec:problem_main_results}, the asymptotic behavior of the fluid flow is governed by the critical geometric parameter $\sigma_{\epsilon}$, defined as:
    \begin{equation}
        \sigma_{\epsilon} := \epsilon \left| \log \frac{a_{\epsilon}}{\epsilon} \right|^{\frac{1}{2}}
    \end{equation}
\end{definition}

The following lemma generalizes the classical results from homogenization theory \cite{Allaire1991,Lu2020} to our specific rectangular Bravais lattice configuration.

\begin{lemma}[Global Poincaré Inequality]\label{lemma:poincare}
    Let $\Omega_m$ be the fluid domain with dense obstacles defined in Section \ref{sec:problem_main_results}, with zero boundary conditions on the holes $\partial T_{\epsilon}^{i,j}$ and the exterior boundary $\partial \Omega$. There exists a constant $C > 0$, independent of $\epsilon$ (and thus independent of the obstacle density), such that for all $\mathbf{v} \in H_0^1(\Omega_m)^d$:
    \begin{equation}
        \|\mathbf{v}\|_{L^2(\Omega_m)} \le C \min\{1, \sigma_{\epsilon}\} \|\nabla \mathbf{v}\|_{L^2(\Omega_m)}.
    \end{equation}
\end{lemma}

\begin{proof}
    We extend $\mathbf{v} \in H_0^1(\Omega_m)^d$ by zero inside the obstacles $T_{\epsilon}^{i,j}$ to obtain $\tilde{\mathbf{v}} \in H_0^1(\Omega)^d$. The classical Poincaré inequality on the macroscopic domain $\Omega$ immediately yields $\|\tilde{\mathbf{v}}\|_{L^2(\Omega)} \le C_P \|\nabla \tilde{\mathbf{v}}\|_{L^2(\Omega)}$, providing the global upper bound $C_P$ independent of $\epsilon$.

    To derive the $\sigma_\epsilon$-dependent local bound, we decompose $\Omega$ into the union of active fluid cells $P_\epsilon^{i,j}$ and a boundary layer $\Lambda_{\epsilon}$ near $\partial \Omega$ containing no complete obstacles. The width of $\Lambda_{\epsilon}$ is of order $\mathcal{O}(\epsilon)$. Since $\mathbf{v} = 0$ on $\partial \Omega$, the standard one-dimensional Poincaré inequality across the thickness of $\Lambda_{\epsilon}$ gives
    \begin{equation}
        \|\mathbf{v}\|_{L^2(\Lambda_{\epsilon})}^2 \le C_0 \epsilon^2 \|\nabla \mathbf{v}\|_{L^2(\Lambda_{\epsilon})}^2.
    \end{equation}

    For each cell $P_\epsilon^{i,j}$, let $B_{\mathrm{out}}^{i,j}$ be its circumscribed ball with radius $R_{\epsilon} = C_{\mathrm{out}} \epsilon$, and $B_{\mathrm{in}}^{i,j} \subset T_{\epsilon}^{i,j}$ be the inscribed ball with radius $r_{\epsilon} = \alpha a_{\epsilon}$ (where $\alpha, C_{\mathrm{out}} > 0$ are geometric lattice constants). Since $\mathbf{v} = 0$ on $\partial T_{\epsilon}^{i,j}$, it vanishes on $\partial B_{\mathrm{in}}^{i,j}$. Using polar coordinates $(r, \theta)$ centered at the obstacle, the fundamental theorem of calculus along radial segments yields
    \begin{equation}
        \mathbf{v}(r, \theta) = \int_{r_{\epsilon}}^r \frac{\partial \mathbf{v}}{\partial \rho}(\rho, \theta) \, \mathrm{d}\rho.
    \end{equation}
    Applying the Cauchy-Schwarz inequality, we separate the derivative and the geometric weight:
    \begin{equation}
        |\mathbf{v}(r, \theta)|^2 \le \left( \int_{r_{\epsilon}}^r \left| \frac{\partial \mathbf{v}}{\partial \rho} \right|^2 \rho \, \mathrm{d}\rho \right) \left( \int_{r_{\epsilon}}^r \frac{1}{\rho} \, \mathrm{d}\rho \right) \le \ln\left(\frac{r}{r_{\epsilon}}\right) \int_{r_{\epsilon}}^{R_{\epsilon}} |\nabla \mathbf{v}|^2 \rho \, \mathrm{d}\rho.
    \end{equation}
    Integrating this point-wise bound over the outer ball $B_{\mathrm{out}}^{i,j}$, we obtain
    \begin{align}
        \|\mathbf{v}\|_{L^2(B_{\mathrm{out}}^{i,j})}^2 = \int_0^{2\pi} \int_{r_{\epsilon}}^{R_{\epsilon}} |\mathbf{v}(r, \theta)|^2 r \, \mathrm{d}r \mathrm{d}\theta & \le \|\nabla \mathbf{v}\|_{L^2(B_{\mathrm{out}}^{i,j})}^2 \int_{r_{\epsilon}}^{R_{\epsilon}} r \ln\left(\frac{r}{r_{\epsilon}}\right) \, \mathrm{d}r \nonumber                                 \\
                                                                                                                                                                      & = \|\nabla \mathbf{v}\|_{L^2(B_{\mathrm{out}}^{i,j})}^2 \left[ \frac{R_{\epsilon}^2}{2} \ln\left(\frac{R_{\epsilon}}{r_{\epsilon}}\right) - \frac{R_{\epsilon}^2 - r_{\epsilon}^2}{4} \right].
    \end{align}
    Dropping the negative term and substituting the radii $R_{\epsilon} = C_{\mathrm{out}}\epsilon$ and $r_{\epsilon} = \alpha a_{\epsilon}$, the geometric bracket is bounded by $\frac{1}{2} C_{\mathrm{out}}^2 \epsilon^2 \ln\left(\frac{C_{\mathrm{out}} \epsilon}{\alpha a_{\epsilon}}\right)$. By defining the characteristic scaling $\sigma_{\epsilon}^2 := \epsilon^2 |\ln(a_{\epsilon}/\epsilon)|$, there exists a constant $C_1 > 0$ such that the integral evaluates to $\le C_1 \sigma_{\epsilon}^2$.

    Summing over all cells $P_\epsilon^{i,j}$ and accounting for the finite overlap of the circumscribed balls, we have $\sum_{i,j} \|\mathbf{v}\|_{L^2(P_\epsilon^{i,j})}^2 \le C_2 \sigma_{\epsilon}^2 \|\nabla \mathbf{v}\|_{L^2(\Omega_m)}^2$. Combining this with the boundary layer bound, we conclude
    \begin{equation}
        \|\mathbf{v}\|_{L^2(\Omega_m)}^2 = \|\mathbf{v}\|_{L^2(\Lambda_{\epsilon})}^2 + \sum_{i,j} \|\mathbf{v}\|_{L^2(P_\epsilon^{i,j})}^2 \le C_3 \sigma_{\epsilon}^2 \|\nabla \mathbf{v}\|_{L^2(\Omega_m)}^2.
    \end{equation}
    Taking the square root and combining this microscopic limit with the macroscopic bound $\|\mathbf{v}\|_{L^2(\Omega_m)} \le C_P \|\nabla \mathbf{v}\|_{L^2(\Omega_m)}$ completes the proof.
\end{proof}

\begin{lemma}[Existence of a Uniformly Bounded Test Pressure]\label{lemma:positivity}
    For the effective fluid domain $\Omega_m$ defined in Section \ref{sec:problem_main_results}, there exists a non-trivial test pressure field $q^* \in Q(\Omega_m) \cap H^1(\Omega_m)$ and a strictly positive constant $C_q > 0$, independent of the obstacle density $m$ (and thus independent of $\epsilon$), such that its gradient is bounded by its norm:
    \begin{equation}
        \|\nabla q^*\|_{L^2(\Omega_m)} \le C_q \|q^*\|_{L^2(\Omega_m)}.
    \end{equation}
\end{lemma}

\begin{proof}
    Let $(x, y) \in \Omega = (0, L_x) \times (0, L_y)$ denote the spatial coordinates. Define the test pressure field $q^*(x, y) = x - \bar{x}$, where $\bar{x} = |\Omega_m|^{-1} \int_{\Omega_m} x \, \mathrm{d}x\mathrm{d}y$ is the spatial mean, ensuring $q^* \in Q(\Omega_m)$ (i.e., zero mean on the fluid domain). Since $\nabla q^* = (1, 0)^\top$, its gradient norm is exactly the measure of the fluid domain:
    \begin{equation}
        \|\nabla q^*\|_{L^2(\Omega_m)}^2 = \int_{\Omega_m} 1 \, \mathrm{d}x\mathrm{d}y = |\Omega_m|.
    \end{equation}

    To bound $\|q^*\|_{L^2(\Omega_m)}$ from below, we evaluate the geometric variance of the $x$-coordinate over the perforated domain $\Omega_m$:
    \begin{equation} \label{eq:norm_variance}
        \|q^*\|_{L^2(\Omega_m)}^2 = \int_{\Omega_m} (x - \bar{x})^2 \, \mathrm{d}x\mathrm{d}y.
    \end{equation}
    By the Bathtub principle (or symmetric decreasing rearrangement), for any domain $\Omega_m \subset \Omega$ with a fixed total area $|\Omega_m|$, the integral in \eqref{eq:norm_variance} is strictly minimized when $\Omega_m$ is rearranged into a single continuous vertical strip centered at $\bar{x}$ with width $w = |\Omega_m| / L_y$. Thus, we obtain the rigorous geometric lower bound:
    \begin{equation}
        \int_{\Omega_m} (x - \bar{x})^2 \, \mathrm{d}x\mathrm{d}y \ge \int_0^{L_y} \int_{-w/2}^{w/2} \xi^2 \, \mathrm{d}\xi \mathrm{d}y = L_y \frac{w^3}{12} = \frac{|\Omega_m|^3}{12 L_y^2}.
    \end{equation}

    Taking the ratio of the gradient norm squared to the $L^2$-norm squared, we immediately obtain a bound dependent purely on the fluid domain measure:
    \begin{equation}
        \frac{\|\nabla q^*\|_{L^2(\Omega_m)}^2}{\|q^*\|_{L^2(\Omega_m)}^2} \le \frac{|\Omega_m|}{|\Omega_m|^3 / (12 L_y^2)} = \frac{12 L_y^2}{|\Omega_m|^2}.
    \end{equation}

    To ensure this bound is uniformly independent of the obstacle density, we assume the maximum solid area fraction is strictly bounded away from complete blockage, denoted by $\phi_{max} \in (0, 1)$. Indeed, $\phi_{max}$ is a fixed geometric constant strictly less than unity, as $T \subset P_0$ and $a_\epsilon/\epsilon \le \kappa$. Physically, this ensures the existence of interconnected fluid channels and prevents the complete blockage of the device. Consequently, the effective fluid area satisfies $|\Omega_m| \ge (1 - \phi_{max})|\Omega| = (1 - \phi_{max}) L_x L_y$. Substituting this macroscopic bound yields:
    \begin{equation}
        \frac{\|\nabla q^*\|_{L^2(\Omega_m)}}{\|q^*\|_{L^2(\Omega_m)}} \le \frac{\sqrt{12} L_y}{|\Omega_m|} \le \frac{\sqrt{12}}{(1 - \phi_{max}) L_x} := C_q.
    \end{equation}
    This completes the proof.
\end{proof}

\begin{lemma}[Properties of the Restriction Operator]\label{lemma:restriction_operator}
    For the effective fluid domain $\Omega_m$ defined in Section \ref{sec:problem_main_results}, there exists a linear restriction operator $R_{\epsilon}: H_0^1(\Omega)^d \to H_0^1(\Omega_m)^d$ satisfying the following fundamental properties:
    \begin{enumerate}
        \item {Preservation of zero-trace functions}: For any $\mathbf{u} \in H_0^1(\Omega_m)^d$, extended by zero to $\Omega$ as $\tilde{\mathbf{u}}$, it holds that $R_{\epsilon}(\tilde{\mathbf{u}}) = \mathbf{u}$ in $\Omega_m$.
        \item {Controlled modification of divergence}: For any $\mathbf{u} \in H_0^1(\Omega)^d$, the divergence of the restricted field in $\Omega_m$ satisfies $\nabla \cdot R_{\epsilon}(\mathbf{u}) = \nabla \cdot \mathbf{u} + c_{i,j}$ within a local control volume $C_{\epsilon}^{i,j}$ surrounding each solid obstacle $T_{\epsilon}^{i,j}$, and $\nabla \cdot R_{\epsilon}(\mathbf{u}) = \nabla \cdot \mathbf{u}$ elsewhere. Crucially, {if $\nabla \cdot \mathbf{u} = 0$ identically inside the solid obstacles, then $\nabla \cdot R_{\epsilon}(\mathbf{u}) = \nabla \cdot \mathbf{u}$ everywhere in $\Omega_m$}.
        \item {Stability estimate}: For any $\mathbf{u} \in H_0^1(\Omega)^d$, the gradient of the restricted field satisfies the bound:
              \begin{equation}
                  \|\nabla R_{\epsilon}(\mathbf{u})\|_{L^2(\Omega_m)} \le C \left( \|\nabla \mathbf{u}\|_{L^2(\Omega)} + (1 + \sigma_{\epsilon}^{-1}) \|\mathbf{u}\|_{L^2(\Omega)} \right)
              \end{equation}
              where $C > 0$ is a constant independent of $\epsilon$.
    \end{enumerate}
\end{lemma}

\begin{proof}[Proof of Lemma \ref{lemma:restriction_operator}]
    We construct the restriction operator $R_{\epsilon}$ by solving local Stokes problems. For each active cell $(i,j) \in \mathcal{I}_{\epsilon}$, let $B_{\epsilon}^{i,j} \subset P_\epsilon^{i,j}$ be an inscribed open ball strictly covering the solid obstacle $T_{\epsilon}^{i,j}$, and define the fluid control volume $C_{\epsilon}^{i,j} = B_{\epsilon}^{i,j} \setminus T_{\epsilon}^{i,j}$.

    For any $\mathbf{u} \in H_0^1(\Omega)^d$, we set $R_{\epsilon}(\mathbf{u})$ piecewise. Away from the obstacles:
    \begin{equation}
        R_{\epsilon}(\mathbf{u}) = \mathbf{u} \quad \text{in } \Omega \setminus \bigcup_{(i,j) \in \mathcal{I}_{\epsilon}} B_{\epsilon}^{i,j},
    \end{equation}
    and inside the solid obstacles:
    \begin{equation}
        R_{\epsilon}(\mathbf{u}) = 0 \quad \text{in } T_{\epsilon}^{i,j}.
    \end{equation}

    Inside $C_{\epsilon}^{i,j}$, we define $R_{\epsilon}(\mathbf{u}) = \mathbf{v}$ as the solution to the local Stokes problem:
    \begin{equation} \label{eq:local_stokes}
        \begin{cases}
            -\Delta \mathbf{v} + \nabla q = -\Delta \mathbf{u}          & \text{in } C_{\epsilon}^{i,j}          \\
            \nabla \cdot \mathbf{v} = \nabla \cdot \mathbf{u} + c_{i,j} & \text{in } C_{\epsilon}^{i,j}          \\
            \mathbf{v} = \mathbf{u}                                     & \text{on } \partial B_{\epsilon}^{i,j} \\
            \mathbf{v} = 0                                              & \text{on } \partial T_{\epsilon}^{i,j}
        \end{cases}
    \end{equation}
    where $c_{i,j}$ is a constant. Well-posedness requires the compatibility condition:
    \begin{equation} \label{eq:compatibility}
        \int_{C_{\epsilon}^{i,j}} (\nabla \cdot \mathbf{v}) \, \mathrm{d}x = \int_{\partial C_{\epsilon}^{i,j}} \mathbf{v} \cdot n \, \mathrm{d}s = \int_{\partial B_{\epsilon}^{i,j}} \mathbf{u} \cdot n_{\mathrm{out}} \, \mathrm{d}s - \int_{\partial T_{\epsilon}^{i,j}} 0 \cdot n_{\mathrm{in}} \, \mathrm{d}s.
    \end{equation}
    Applying the divergence theorem to $\mathbf{u}$ over $B_{\epsilon}^{i,j}$ yields:
    \begin{equation} \label{eq:divergence_ball}
        \int_{\partial B_{\epsilon}^{i,j}} \mathbf{u} \cdot n_{\mathrm{out}} \, \mathrm{d}s = \int_{B_{\epsilon}^{i,j}} \nabla \cdot \mathbf{u} \, \mathrm{d}x = \int_{C_{\epsilon}^{i,j}} \nabla \cdot \mathbf{u} \, \mathrm{d}x + \int_{T_{\epsilon}^{i,j}} \nabla \cdot \mathbf{u} \, \mathrm{d}x.
    \end{equation}
    Substituting \eqref{eq:divergence_ball} into \eqref{eq:compatibility} and using the divergence equation from \eqref{eq:local_stokes}, we obtain:
    \begin{equation}
        \int_{C_{\epsilon}^{i,j}} (\nabla \cdot \mathbf{u} + c_{i,j}) \, \mathrm{d}x = \int_{C_{\epsilon}^{i,j}} \nabla \cdot \mathbf{u} \, \mathrm{d}x + \int_{T_{\epsilon}^{i,j}} \nabla \cdot \mathbf{u} \, \mathrm{d}x.
    \end{equation}
    This uniquely determines $c_{i,j}$:
    \begin{equation}
        c_{i,j} = \frac{1}{|C_{\epsilon}^{i,j}|} \int_{T_{\epsilon}^{i,j}} \nabla \cdot \mathbf{u} \, \mathrm{d}x,
    \end{equation}
    which establishes Property 2.

    For Property 1, if $\mathbf{u} \in H_0^1(\Omega_m)^d$ is extended by zero to $\Omega$, then $\mathbf{u} = 0$ on $\partial T_{\epsilon}^{i,j}$ and $\nabla \cdot \mathbf{u} = 0$ identically in $T_{\epsilon}^{i,j}$. Consequently, $c_{i,j} = 0$, and \eqref{eq:local_stokes} admits the unique solution $\mathbf{v} = \mathbf{u}$ in $C_{\epsilon}^{i,j}$.

    For Property 3, we rescale the physical ball $B_{\epsilon}^{i,j}$ to a unit ball $\hat{B}$ via $y = (x - x_{i,j})/\epsilon$. The obstacle transforms into $\hat{T}_{\eta}$ with characteristic size $\eta = a_{\epsilon}/\epsilon$, defining the reference annulus $\hat{C}_{\eta} = \hat{B} \setminus \hat{T}_{\eta}$. Defining $\hat{\mathbf{u}}(y) = \mathbf{u}(x_{i,j} + \epsilon y)$ and $\hat{\mathbf{v}}(y) = \mathbf{v}(x_{i,j} + \epsilon y)$, standard Bogovskii estimates \cite[Lemma 2.2.5]{Allaire1991} yield:
    \begin{equation}
        \|\hat{\nabla} \hat{\mathbf{v}}\|_{L^2(\hat{C}_{\eta})}^2 \le C \left( \|\hat{\nabla} \hat{\mathbf{u}}\|_{L^2(\hat{B})}^2 + |\log \eta|^{-1} \|\hat{\mathbf{u}}\|_{L^2(\hat{B})}^2 \right).
    \end{equation}

    Scaling back to physical variables, we use $\|\hat{\nabla} \hat{\mathbf{v}}\|_{L^2(\hat{C}_{\eta})}^2 = \|\nabla \mathbf{v}\|_{L^2(C_{\epsilon}^{i,j})}^2$, $\|\hat{\nabla} \hat{\mathbf{u}}\|_{L^2(\hat{B})}^2 = \|\nabla \mathbf{u}\|_{L^2(B_{\epsilon}^{i,j})}^2$, and $\|\hat{\mathbf{u}}\|_{L^2(\hat{B})}^2 = \epsilon^{-2} \|\mathbf{u}\|_{L^2(B_{\epsilon}^{i,j})}^2$:
    \begin{equation}
        \|\nabla \mathbf{v}\|_{L^2(C_{\epsilon}^{i,j})}^2 \le C \left( \|\nabla \mathbf{u}\|_{L^2(B_{\epsilon}^{i,j})}^2 + \frac{|\log \eta|^{-1}}{\epsilon^2} \|\mathbf{u}\|_{L^2(B_{\epsilon}^{i,j})}^2 \right).
    \end{equation}
    Recalling $\eta = a_{\epsilon}/\epsilon$ and $\sigma_{\epsilon} = \epsilon |\log(a_{\epsilon}/\epsilon)|^{1/2}$, the penalty coefficient scales as:
    \begin{equation}
        \frac{|\log \eta|^{-1}}{\epsilon^2} \propto \frac{|\log(a_{\epsilon}/\epsilon)|^{-1}}{\epsilon^2} = \sigma_{\epsilon}^{-2},
    \end{equation}
    yielding the local estimate:
    \begin{equation}
        \|\nabla \mathbf{v}\|_{L^2(C_{\epsilon}^{i,j})}^2 \le C \left( \|\nabla \mathbf{u}\|_{L^2(B_{\epsilon}^{i,j})}^2 + \sigma_{\epsilon}^{-2} \|\mathbf{u}\|_{L^2(B_{\epsilon}^{i,j})}^2 \right).
    \end{equation}

    Summing over all active cells $(i,j) \in \mathcal{I}_{\epsilon}$ and noting the topologically bounded finite overlap of $B_{\epsilon}^{i,j}$, we obtain the global bound:
    \begin{equation}
        \|\nabla R_{\epsilon}(\mathbf{u})\|_{L^2(\Omega_m)}^2 \le C' \left( \|\nabla \mathbf{u}\|_{L^2(\Omega)}^2 + \sigma_{\epsilon}^{-2} \|\mathbf{u}\|_{L^2(\Omega)}^2 \right).
    \end{equation}
    Taking the square root and applying $\sqrt{a^2 + b^2} \le a + b$ yields the final stability estimate:
    \begin{equation}
        \|\nabla R_{\epsilon}(\mathbf{u})\|_{L^2(\Omega_m)} \le C'' \left( \|\nabla \mathbf{u}\|_{L^2(\Omega)} + \sigma_{\epsilon}^{-1} \|\mathbf{u}\|_{L^2(\Omega)} \right),
    \end{equation}
    completing the proof.
\end{proof}

\section{Proofs of the Main Theorems}\label{sec:proofs_main}

With the foundational geometric and functional estimates established in Section \ref{sec:proofs}, we now present the rigorous proofs for the sharp pre-asymptotic degradation rate of the pressure-velocity coupling stability in the obstacle array domain $\Omega_m$.

\subsection{Proof of Theorem \ref{thm:main_scaling}: The Upper Bound}

\begin{proof}[Proof of the upper bound]
    The continuous inf-sup constant for the domain $\Omega_m$ is defined by the variational problem:
    \begin{equation} \label{eq:infsup_def_cp}
        \beta(\Omega_m) = \inf_{q \in L_0^2(\Omega_m) \setminus \{0\}} \sup_{\mathbf{v} \in H_0^1(\Omega_m)^d \setminus \{\mathbf{0}\}} \frac{\int_{\Omega_m} q \nabla \cdot \mathbf{v} \, \mathrm{d}x}{\|q\|_{L^2(\Omega_m)} \|\nabla \mathbf{v}\|_{L^2(\Omega_m)}}.
    \end{equation}
    To establish an upper bound, it suffices to evaluate the supremum quotient for a specific test pressure field. We select the global linear pressure field $q^*(x, y) = x - \bar{x} \in Q(\Omega_m)$ constructed in Lemma \ref{lemma:positivity}.

    By integration by parts, we get the following estimate.
    \begin{equation}
        \int_{\Omega_m} q^* \nabla \cdot \mathbf{v} \, \mathrm{d}x = -\int_{\Omega_m} \nabla q^* \cdot \mathbf{v} \, \mathrm{d}x \le \|\nabla q^*\|_{L^2(\Omega_m)} \|\mathbf{v}\|_{L^2(\Omega_m)} \le C_P \sigma_{\epsilon} \|\nabla q^*\|_{L^2(\Omega_m)} \|\nabla \mathbf{v}\|_{L^2(\Omega_m)},
    \end{equation}
    where we used the global Poincaré inequality from Lemma \ref{lemma:poincare}.
    We insert this bound into the quotient in \eqref{eq:infsup_def} evaluated at $q^*$:
    \begin{equation}
        \beta(\Omega_m) \le \sup_{\mathbf{v} \in H_0^1(\Omega_m)^d} \frac{C_P \sigma_{\epsilon} \|\nabla q^*\|_{L^2(\Omega_m)} \|\nabla \mathbf{v}\|_{L^2(\Omega_m)}}{\|q^*\|_{L^2(\Omega_m)} \|\nabla \mathbf{v}\|_{L^2(\Omega_m)}} = C_P \sigma_{\epsilon} \frac{\|\nabla q^*\|_{L^2(\Omega_m)}}{\|q^*\|_{L^2(\Omega_m)}}.
    \end{equation}
    By Lemma \ref{lemma:positivity}, the ratio $\|\nabla q^*\|_{L^2(\Omega_m)} / \|q^*\|_{L^2(\Omega_m)}$ is bounded by a constant $C_q$ strictly independent of the obstacle density. Therefore:
    \begin{equation}
        \beta(\Omega_m) \le (C_P C_q) \sigma_{\epsilon}.
    \end{equation}
    For the specific case where the obstacle size scales proportionally with the cell size ($a_{\epsilon} = \kappa \epsilon$), the parameter simplifies to $\sigma_{\epsilon} = \epsilon |\log \kappa|^{1/2} = \mathcal{O}(\epsilon)$. Since the number of cells per dimension is $m \propto \epsilon^{-1}$, we conclude that $\beta(\Omega_m) \le \mathcal{O}(m^{-1})$, completing the proof.
\end{proof}

\subsection{Proof of Theorem \ref{thm:main_scaling}: The Lower Bound and Sharpness}

\begin{proof}[Proof of lower bound]
    Proving a lower bound for the continuous inf-sup constant $\beta(\Omega_m) \ge C^*$ is equivalent to showing that, for any $q \in L_0^2(\Omega_m)$, there exists a velocity field $\mathbf{v} \in H_0^1(\Omega_m)^d$ such that $\nabla \cdot \mathbf{v} = q$ and $\|\nabla \mathbf{v}\|_{L^2(\Omega_m)} \le (C^*)^{-1} \|q\|_{L^2(\Omega_m)}$.

    Given $q \in L_0^2(\Omega_m)$, we extend it by zero to the entire macroscopic domain $\Omega$ to define $\tilde{q}$. Since the extension is by zero, $\tilde{q} \in L_0^2(\Omega)$ and $\|\tilde{q}\|_{L^2(\Omega)} = \|q\|_{L^2(\Omega_m)}$. Because the macroscopic domain $\Omega$ is fixed and independent of $\epsilon$, standard estimate guarantees the existence of a global velocity field $\mathbf{V} \in H_0^1(\Omega)^d$ satisfying:
    \begin{equation}
        \nabla \cdot \mathbf{V} = \tilde{q} \quad \text{in } \Omega, \quad \text{and} \quad \|\nabla \mathbf{V}\|_{L^2(\Omega)} \le C_B \|\tilde{q}\|_{L^2(\Omega)} = C_B \|q\|_{L^2(\Omega_m)},
    \end{equation}
    where $C_B$ is the standard LBB constant for the unperforated domain $\Omega$.

    To enforce the no-slip condition on the obstacle boundaries while preserving the target divergence, we apply the restriction operator constructed in Lemma \ref{lemma:restriction_operator}, defining $\mathbf{v} = R_{\epsilon}(\mathbf{V}) \in H_0^1(\Omega_m)^d$.

    According to Property 2 of Lemma \ref{lemma:restriction_operator}, because $\nabla \cdot \mathbf{V} = \tilde{q} = 0$ identically inside all solid obstacles $T_{\epsilon}^{i,j}$, the restriction operator perfectly preserves the divergence field globally without requiring local corrections (i.e., $c_{i,j} = 0$). Thus, $\nabla \cdot \mathbf{v} = \nabla \cdot \mathbf{V} = q$ everywhere in $\Omega_m$.

    To evaluate the norm of the constructed field, we invoke the stability estimate (Property 3) of Lemma \ref{lemma:restriction_operator}:
    \begin{equation}
        \|\nabla \mathbf{v}\|_{L^2(\Omega_m)} \le C_R \left( \|\nabla \mathbf{V}\|_{L^2(\Omega)} + (1 + \sigma_{\epsilon}^{-1}) \|\mathbf{V}\|_{L^2(\Omega)} \right).
    \end{equation}
    We bound the $L^2$-norm of the macroscopic field $V$ using the standard Poincaré inequality on the fixed domain $\Omega$, yielding $\|V\|_{L^2(\Omega)} \le C_0 \|\nabla V\|_{L^2(\Omega)}$. Substituting this and the global infsup constant $C_B$ into the stability estimate yields:
    \begin{align} \label{eq:lower_bound_est}
        \|\nabla \mathbf{v}\|_{L^2(\Omega_m)} & \le C_R \left( 1 + C_0(1 + \sigma_{\epsilon}^{-1}) \right) \|\nabla \mathbf{V}\|_{L^2(\Omega)} \nonumber \\
                                              & \le C_R C_B \left( 1 + C_0 + C_0 \sigma_{\epsilon}^{-1} \right) \|q\|_{L^2(\Omega_m)}.
    \end{align}
    For a small $\epsilon$ (and consequently a large $\sigma_{\epsilon}^{-1}$), the scaling term $\sigma_{\epsilon}^{-1}$ dominates the constant terms. Therefore, there exists an asymptotic constant $C' > 0$ such that:
    \begin{equation}
        \|\nabla \mathbf{v}\|_{L^2(\Omega_m)} \le C' \sigma_{\epsilon}^{-1} \|q\|_{L^2(\Omega_m)}.
    \end{equation}
    This establishes that the norm of the continuous Bogovskii operator on the perforated domain $\Omega_m$ is bounded from above by $\mathcal{O}(\sigma_{\epsilon}^{-1})$. Inverting this norm bound immediately provides the lower bound for the inf-sup constant:
    \begin{equation}
        \beta(\Omega_m) = \inf_{q} \sup_{\mathbf{v}} \frac{\int_{\Omega_m} q \nabla \cdot \mathbf{v}}{\|q\| \|\nabla \mathbf{v}\|} \ge \frac{1}{C' \sigma_{\epsilon}^{-1}} = C'' \sigma_{\epsilon}.
    \end{equation}
    Combined with the upper bound $\beta(\Omega_m) \le (C_P C_q) \sigma_{\epsilon}$ from Theorem \ref{thm:main_scaling}, we conclude the sharp estimate $\beta(\Omega_m) = \Theta(\sigma_{\epsilon})$. For $a_{\epsilon} \propto \epsilon$, this yields $\beta(\Omega_m) = \Theta(\epsilon) = \Theta(m^{-1})$, completing the proof.
\end{proof}

\begin{remark}[Discrete inf-sup Degradation] \label{rmk:discrete_infsup}
    The continuous degradation rate $\beta(\Omega_m) = \Theta(m^{-1})$ is directly inherited by the discrete saddle-point system. For stable conforming mixed finite elements (e.g., Taylor-Hood) admitting a mesh-independent Fortin operator, the discrete inf-sup constant is bounded below by $\beta_h \gtrsim  \beta(\Omega_m) = \Theta(m^{-1})$ \cite{boffi2013mixed,Chen2014}. Conversely, provided the discrete pressure space $Q_h$ contains global linear polynomials, evaluating the discrete inf-sup quotient with $q^*(x,y)$ yields the sharp upper bound $\beta_h \le \mathcal{O}(m^{-1})$. Thus, for standard stable discretizations, the discrete inf-sup constant degradation remains $\beta_h = \Theta(m^{-1})$. Detailed proofs are omitted as they follow straightforwardly from standard mixed finite element theory.
\end{remark}

\section{Numerical Consequences of the Sharp inf-sup Bound}\label{sec:numerical}
In this section, we present a series of numerical experiments designed to verify the sharp pre-asymptotic estimate of the continuous LBB constant established in Theorem \ref{thm:main_scaling} and to map this geometric penalty onto its two practical manifestations: the amplification of the a priori discretization error and the stagnation of standard iterative solvers. We then evaluate the robustness of the proposed adaptively scaled Augmented Lagrangian (AL) preconditioner as the pillar density $m$ increases.

\subsection{Discrete Formulation and Matrix Definitions}

To strictly isolate the continuous, geometry-induced inf-sup degradation from any artificial numerical instabilities, we employ the standard Galerkin finite element method using the inf-sup stable Taylor-Hood element pair ($\mathbb{P}_2-\mathbb{P}_1$). This classical choice guarantees that the discrete LBB condition is inherently satisfied, ensuring that any spectral deterioration observed in the algebraic system originates purely from the micro-structural topology rather than spatial discretization artifacts.

Let $\mathcal{T}_h$ be a shape-regular finite element triangulation of the domain. We define the conforming discrete velocity and pressure spaces as:
\begin{align}
    V_h & = \{ \mathbf{v}_h \in C^0(\overline{\Omega})^d : \mathbf{v}_h|_K \in [\mathbb{P}_2(K)]^d, \forall K \in \mathcal{T}_h \} \cap W_0^{1,2}(\Omega)^d, \\
    Q_h & = \{ q_h \in C^0(\overline{\Omega}) : q_h|_K \in \mathbb{P}_1(K), \forall K \in \mathcal{T}_h \} \cap L_0^2(\Omega),
\end{align}
where $d \in \{2, 3\}$ is the spatial dimension.

The discrete variational formulation seeks a velocity-pressure pair $(\mathbf{u}_h, p_h) \in V_h \times Q_h$ such that:
\begin{align}
    a(\mathbf{u}_h, \mathbf{v}_h) + b(\mathbf{v}_h, p_h) & = (f, \mathbf{v}_h), \quad \forall \mathbf{v}_h \in V_h, \\
    b(\mathbf{u}_h, q_h)                                 & = 0, \quad \forall q_h \in Q_h.
\end{align}
Here, $b(\mathbf{v}, q) = -\int q \nabla \cdot \mathbf{v} \, \mathrm{d}x$ represents the divergence constraint, and the bilinear form $a(\mathbf{u}, \mathbf{v}) = \int \mu \nabla \mathbf{u} : \nabla \mathbf{v} \, \mathrm{d}x$ encapsulates the viscous dissipation.

Upon choosing standard nodal basis functions for the spaces $V_h$ and $Q_h$, this formulation leads to the classical generalized saddle-point algebraic system:
\begin{equation} \label{eq:saddle_point_system}
    \mathcal{A} \mathbf{x} = \mathbf{b} \quad \implies \quad
    \begin{pmatrix}
        A & B^T \\
        B & 0
    \end{pmatrix}
    \begin{pmatrix}
        \mathbf{u} \\
        \mathbf{p}
    \end{pmatrix}
    =
    \begin{pmatrix}
        \mathbf{f} \\
        \mathbf{0}
    \end{pmatrix},
\end{equation}
where $\mathbf{u} \in \mathbb{R}^{n_u}$ and $\mathbf{p} \in \mathbb{R}^{n_p}$ denote the vectors of velocity and pressure degrees of freedom, respectively. The global block matrices are defined as follows:
\begin{itemize}
    \item $A \in \mathbb{R}^{n_u \times n_u}$: the symmetric positive-definite velocity stiffness matrix representing the discrete operator associated with $a(\mathbf{u}, \mathbf{v})$. In the context of our perforated domain analysis, the condition number of this block inherently depends on the obstacle density.
    \item $B \in \mathbb{R}^{n_p \times n_u}$: the discrete negative divergence matrix associated with $b(\mathbf{v}, q)$.
    \item $B^T \in \mathbb{R}^{n_u \times n_p}$: the discrete gradient matrix.
\end{itemize}

With this algebraic framework established, the pressure Schur complement, explicitly defined as $S = B A^{-1} B^T$, becomes the central object of our pre-asymptotic stability analysis.

The primary objective of this section is twofold: first, to numerically validate the exact $\Theta(m^{-1})$ pre-asymptotic degradation rate of the discrete inf-sup constant established in our theoretical analysis, which translates algebraically to the minimal generalized eigenvalue of the scaled Schur complement; second, to inspect the behavior of the discrete solution as the pillar density increases, which is dominated not only by the quasi-optimal convergence, but also by the inf-sup constant; and third, to demonstrate how this analytical bound provides a rigorous foundation for choosing the penalty parameter in Augmented Lagrangian (AL) stabilization.

\subsection{Discrete inf-sup Constant Degrades as $\Theta(m^{-1})$}\label{subsec:validation_infsup}

To verify Theorem \ref{thm:main_scaling}, we compute the discrete inf-sup constant $\beta_h$ for a sequence of domains $\Omega_m$ with increasing pillar density $m$. The discrete constant is evaluated by solving the generalized eigenvalue problem $B A^{-1} B^T q = \lambda M_p q$, where $M_p$ is the pressure mass matrix, and extracting the square root of the smallest non-zero eigenvalue: $\beta_h = \sqrt{\lambda_{min \neq 0}}$.

We consider a unit global bounding channel $\Omega = (0,1)^2$ and partition it into an $m \times m$ unshifted Bravais lattice ($\delta = 0$). Circular pillars with a fixed packing ratio $\kappa = 0.25$ are placed at the center of each active cell. We utilize standard Taylor-Hood ($P_2-P_1$) finite elements, ensuring that the mesh size $h$ is refined sufficiently at each density level to capture the local shear layers.

\begin{figure}[htbp]
    \centering
    \includegraphics[width=0.6\textwidth]{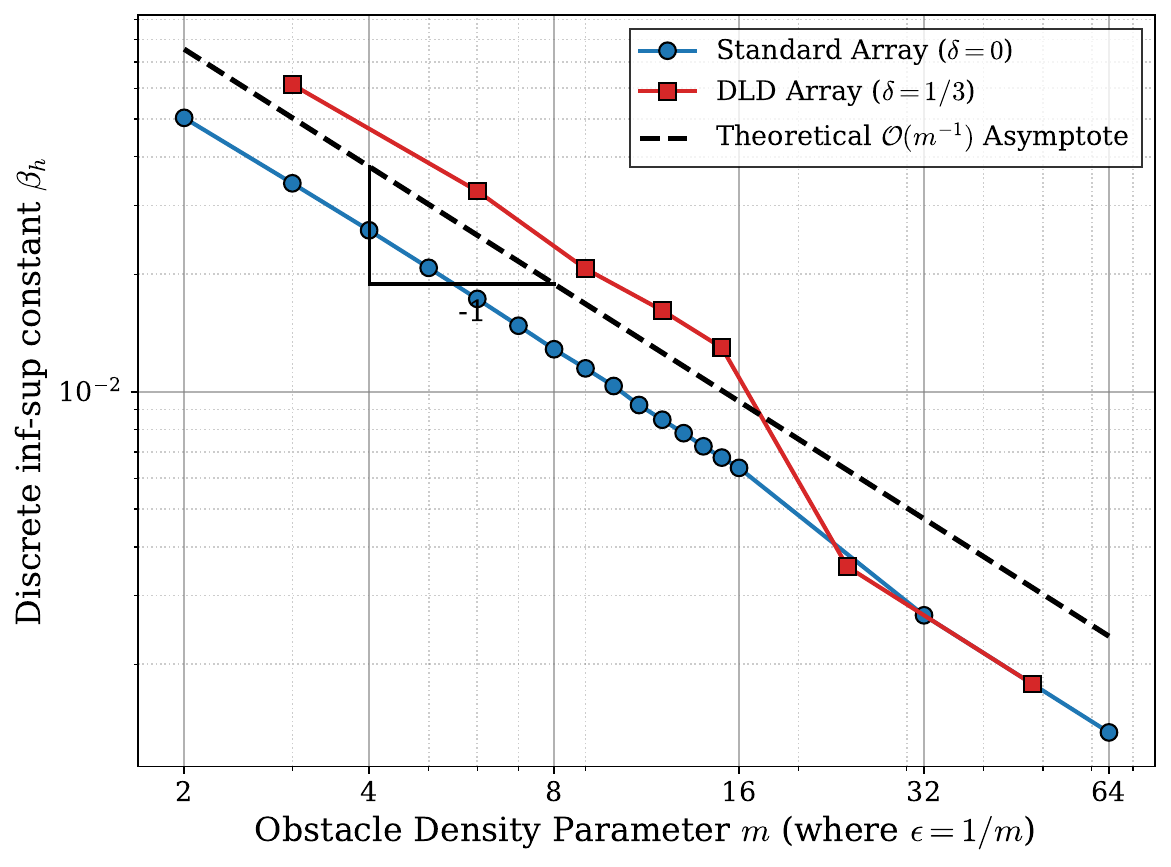}
    \caption{Log-log plot of the discrete inf-sup constant $\beta_h$ as a function of the pillar density $m$. The computed values precisely track the theoretical $\Theta(m^{-1})$ asymptote (indicated by the reference line with slope $-1$), confirming the sharpness of the pre-asymptotic degradation.}
    \label{fig:infsup_degradation}
\end{figure}

As illustrated in Figure \ref{fig:infsup_degradation}, the computed discrete constants align with the theoretical $\Theta(m^{-1})$ trajectory. This numerical evidence corroborates our proof that the deterioration of the saddle-point stability is governed by the increasing pillar density, and that this geometric penalty is captured by the discrete inf-sup constant.

The discrete inf-sup constant tracks the theoretical $\Theta(m^{-1})$ slope across all tested densities. The geometric penalty is fully inherited by the discrete system; it is not an artifact of the discretization.

\subsection{inf-sup Degradation Amplifies Pressure Error Under Dirichlet Constraints}\label{subsec:validation_convergence}

\subsubsection{Experimental Setup and Baseline Convergence}

To investigate the implications of the stability degradation $\beta_h = \Theta(m^{-1})$ established in our theoretical analysis, we design two strictly controlled physical experiments. We consider Stokes flow through a deterministic lateral displacement (DLD) micro-pillar array without volumetric body forces ($\mathbf{f} = \mathbf{0}$). Let $\Gamma_{\mathrm{in}}$, $\Gamma_{\mathrm{out}}$, and $\Gamma_{\mathrm{wall}}$ denote the inlet, outlet, and no-slip boundary (including channel walls and micro-pillars), respectively. The computational domain is defined in Figure \ref{fig:exp_setup}.
To comprehensively test the system, we investigate two distinct macroscopic driving mechanisms:

\begin{enumerate}
    \item \textbf{Setup A: Pressure-Driven Flow.} We impose a macroscopic pressure drop ($\Delta p = m^2$) via normal stress Neumann boundary conditions at the inlet and outlet:
          \begin{equation} \label{eq:bc_pressure}
              \begin{aligned}
                  (-p \mathbf{I} + \nabla \mathbf{u}) \cdot \mathbf{n} & = -p_{\mathrm{in}} \mathbf{n} \quad \text{on } \Gamma_{\mathrm{in}},   \\
                  (-p \mathbf{I} + \nabla \mathbf{u}) \cdot \mathbf{n} & = -p_{\mathrm{out}} \mathbf{n} \quad \text{on } \Gamma_{\mathrm{out}},
              \end{aligned}
          \end{equation}
          where $p_{\mathrm{in}} = m^2$, $p_{\mathrm{out}} = 0.0$, and $\mathbf{n}$ is the outward unit normal vector.

    \item \textbf{Setup B: Velocity-Driven Pipe Flow.} We impose a fixed macroscopic parabolic velocity profile at both the inlet and the outlet (pure Dirichlet conditions) to enforce mass conservation:
          \begin{equation} \label{eq:bc_velocity}
              \begin{aligned}
                  \mathbf{u} & = \mathbf{u}_{D}(y) \quad \text{on } \Gamma_{\mathrm{in}} \cup \Gamma_{\mathrm{out}},
              \end{aligned}
          \end{equation}
          where $\mathbf{u}_{D}(y) = \left[ 4 U_{\max} y (L_y - y) / L_y^2, \, 0 \right]^T$ with $U_{\max} = 1.0$.
\end{enumerate}
\begin{figure}
    \centering
    \includegraphics[width=0.6\textwidth]{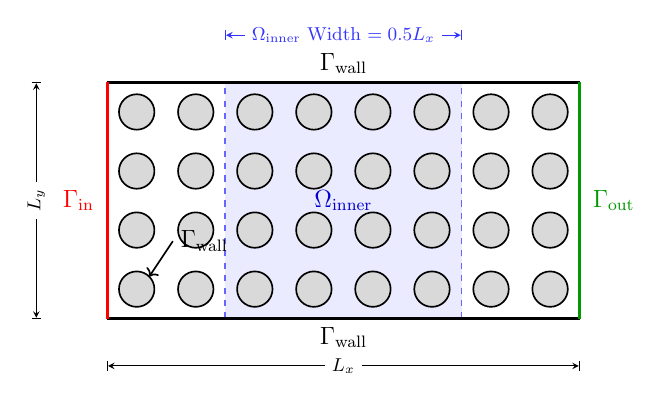}
    \caption{Schematic of the experimental setup for the DLD micro-pillar array. The computational domain $\Omega$ is defined as $(0, L_x) \times (0, L_y)$, where $L_x = 2.0$ and $L_y = 1.0$. The array consists of $2m \times m$ unit cells, each containing a circular pillar of radius $r = 0.3 \epsilon$ located at $(\epsilon(i+0.5), \epsilon(j+0.5))$ for $i = 0, 1, \dots, 2m-1$ and $j = 0, 1, \dots, m-1$, with $\epsilon = 1/m$. The interior evaluation subdomain $\Omega_{\mathrm{inner}}$ is defined as $(0.25L_x, 0.75L_x) \times (0, L_y)$.}
    \label{fig:exp_setup}
\end{figure}

In both cases, standard no-slip Dirichlet conditions are strictly applied to all physical obstacles and bounding walls:
\begin{equation} \label{eq:bc_wall}
    \mathbf{u} = \mathbf{0} \quad \text{on } \Gamma_{\mathrm{wall}}.
\end{equation}

To isolate the discretization error induced by the degraded inf-sup stability from geometric approximation errors (i.e., variational crimes \cite{brenner2008mathematical}), we adopt a hierarchical polynomial elevation strategy \cite{Bank1985}. The reference true solution $(\mathbf{u}_{\mathrm{ref}}, p_{\mathrm{ref}})$ is computed using the high-order $P_3-P_2$ element on the exactly identical second-order isoparametric curved mesh as the targeted Taylor-Hood ($P_2-P_1$) approximation $(\mathbf{u}_h, p_h)$. Relative and absolute errors are strictly evaluated within an interior macroscopic subdomain $\Omega_{\mathrm{inner}} := (0.25L_x, 0.75L_x) \times (0, L_y)$ to bypass corner singularities.

Specifically, the relative errors for the velocity $L^2$-norm, velocity $H^1$-seminorm, and pressure $L^2$-norm are defined respectively as:
\begin{equation} \label{eq:error_velocity}
    e_{\mathbf{u}, L^2} = \frac{\|\mathbf{u}_h - \mathbf{u}_{\mathrm{ref}}\|_{L^2(\Omega_{\mathrm{inner}})}}{\|\mathbf{u}_{\mathrm{ref}}\|_{L^2(\Omega_{\mathrm{inner}})}}, \quad
    e_{\mathbf{u}, H^1} = \frac{\|\nabla \mathbf{u}_h - \nabla \mathbf{u}_{\mathrm{ref}}\|_{L^2(\Omega_{\mathrm{inner}})}}{\|\nabla \mathbf{u}_{\mathrm{ref}}\|_{L^2(\Omega_{\mathrm{inner}})}},
\end{equation}
\begin{equation} \label{eq:error_pressure}
    e_{p, L^2} = \frac{\|(p_h - \bar{p}_h) - (p_{\mathrm{ref}} - \bar{p}_{\mathrm{ref}})\|_{L^2(\Omega_{\mathrm{inner}})}}{\|p_{\mathrm{ref}} - \bar{p}_{\mathrm{ref}}\|_{L^2(\Omega_{\mathrm{inner}})}},
\end{equation}
where $\bar{p} = \frac{1}{|\Omega_{\mathrm{inner}}|} \int_{\Omega_{\mathrm{inner}}} p \, \mathrm{d}x$ denotes the spatial mean of the pressure over the evaluation subdomain. This zero-mean adjustment for the pressure field ensures a rigorous comparison by explicitly filtering out the arbitrary constant associated with the pure Dirichlet boundary conditions in Setup B.

To systematically control the mesh resolution across varying obstacle densities, we introduce a dimensionless {grid factor}, $N_g$. This factor dictates the local resolution within each microscopic unit cell, enforcing the global maximum mesh size to scale as $h_{\max} \approx \epsilon / N_g = (m \cdot N_g)^{-1}$, where $\epsilon = 1/m$ is the characteristic pore size.

As demonstrated in Validation 1 (Table \ref{tab:baseline_convergence}), for a fixed obstacle density ($m=4$), successive mesh refinements yield the optimal asymptotic convergence rates regardless of the driving boundary conditions. The relative pressure error decays at $\mathcal{O}(h^2)$, while the velocity $L^2$ and $H^1$ errors decay at $\mathcal{O}(h^3)$ and $\mathcal{O}(h^2)$ respectively. This confirms the baseline consistency of the spatial discretization when the macroscopic structure is invariant.

\begin{table}[htbp]
    \centering
    \caption{Validation 1: Baseline convergence with a fixed obstacle density ($m=4$). Optimal Taylor-Hood convergence rates are achieved under both pressure-driven and velocity-driven boundary conditions as $h \to 0$.}
    \label{tab:baseline_convergence}
    \renewcommand{\arraystretch}{1.2}
    \resizebox{\textwidth}{!}{
        \begin{tabular}{ccccccc}
            \toprule
            \multicolumn{7}{c}{\textbf{Setup A: Pressure-Driven Boundary Conditions}}                                                                                     \\
            \midrule
            $h_{\max}$ & Rel. $p$ ($L^2$) & Order ($p$) & Rel. $\mathbf{u}$ ($L^2$) & Order ($\mathbf{u}_{L^2}$) & Rel. $\mathbf{u}$ ($H^1$) & Order ($\mathbf{u}_{H^1}$) \\
            \midrule
            0.0786     & 1.90e-02         & --          & 3.55e-02                  & --                         & 1.70e-01                  & --                         \\
            0.0549     & 7.85e-03         & 2.46        & 1.22e-02                  & 2.98                       & 9.81e-02                  & 1.53                       \\
            0.0426     & 4.71e-03         & 2.02        & 6.42e-03                  & 2.53                       & 6.56e-02                  & 1.58                       \\
            0.0338     & 3.25e-03         & 1.60        & 3.89e-03                  & 2.16                       & 4.75e-02                  & 1.40                       \\
            0.0284     & 2.17e-03         & 2.31        & 2.27e-03                  & 3.10                       & 3.40e-02                  & 1.91                       \\
            0.0247     & 1.55e-03         & 2.41        & 1.42e-03                  & 3.35                       & 2.52e-02                  & 2.14                       \\
            0.0218     & 1.18e-03         & 2.18        & 9.83e-04                  & 2.94                       & 1.99e-02                  & 1.89                       \\
            \midrule
            \multicolumn{7}{c}{\textbf{Setup B: Velocity-Driven Pipe Flow Boundary Conditions}}                                                                           \\
            \midrule
            0.0786     & 1.90e-02         & --          & 3.55e-02                  & --                         & 1.70e-01                  & --                         \\
            0.0549     & 8.48e-03         & 2.25        & 1.19e-02                  & 3.05                       & 9.83e-02                  & 1.53                       \\
            0.0426     & 5.24e-03         & 1.89        & 6.15e-03                  & 2.59                       & 6.57e-02                  & 1.58                       \\
            0.0338     & 3.50e-03         & 1.74        & 3.74e-03                  & 2.15                       & 4.75e-02                  & 1.40                       \\
            0.0284     & 2.29e-03         & 2.44        & 2.20e-03                  & 3.06                       & 3.40e-02                  & 1.92                       \\
            0.0247     & 1.60e-03         & 2.55        & 1.38e-03                  & 3.34                       & 2.52e-02                  & 2.15                       \\
            0.0218     & 1.21e-03         & 2.25        & 9.57e-04                  & 2.91                       & 1.99e-02                  & 1.89                       \\
            \bottomrule
        \end{tabular}
    }
\end{table}

\subsubsection{Amplification Penalty and Scaling Analysis}

The pronounced degradation of the convergence emerges in Validation 2 (Tables \ref{tab:error_amplification} and \ref{tab:error_amplification_orders}). Here, we fix the local mesh resolution within each pore space ($N_g=6$) and increase the obstacle density $m$. This strictly forces the global maximum mesh size to scale inversely with the density ($h_{\max} \propto m^{-1}$).

To prove that the observed loss of optimal convergence is not a numerical artifact but a fundamental mathematical manifestation of the multiscale saddle-point system, we analyze the error in three steps, mapping our theoretical bounds directly to the empirical data in Table \ref{tab:error_amplification}.

\paragraph{Step 1: Absolute Scaling of the Physical Fields}
Before analyzing the errors, we must establish the physical baselines driven by different boundary conditions. Crucially, we designed both setups to maintain equivalent macroscopic scalings, isolating the boundary formulation as the sole independent variable.
\begin{itemize}
    \item \textbf{Setup A (Pressure-Driven):} By imposing a macroscopic pressure drop that explicitly scales with the characteristic hydraulic resistance ($\Delta p = m^2$), we artificially drive a macroscopically stable flow rate. According to Darcy's law, the absolute velocity remains strictly $\mathcal{O}(1)$, while the reference pressure scales as $\mathcal{O}(m^2)$. This is verified in Table \ref{tab:error_amplification} (Setup A), where $\|\mathbf{u}_{\mathrm{ref}}\|_{L^2}$ remains remarkably stable at $\approx 8.0 \times 10^{-3}$ across all densities up to $m=128$, and $\|p_{\mathrm{ref}}\|_{L^2}$ scales strictly quadratically, reaching $2.00 \times 10^3$ at $m=128$.
    \item \textbf{Setup B (Velocity-Driven):} Forcing a constant macroscopic flow rate ($\|\mathbf{u}_{\mathrm{ref}}\|_{L^2} \approx 1.0$) through densely packed pillars similarly reflects Darcy's law. The required absolute pressure must identically grow quadratically: $\|p_{\mathrm{ref}}\|_{L^2} = \mathcal{O}(m^2)$. Table \ref{tab:error_amplification} confirms this: as $m$ doubles iteratively up to 128, $\|p_{\mathrm{ref}}\|_{L^2}$ quadruples at each step, culminating at $2.44 \times 10^5$.
\end{itemize}

\paragraph{Step 2: Invariance of the Relative Velocity Error}
In both physical regimes, the spatial derivatives of the highly localized velocity field scale strictly with the microscopic length ($\epsilon \propto m^{-1}$):
\begin{equation} \label{eq:vel_scaling}
    |\mathbf{u}|_{H^k} \sim \epsilon^{-k} \|\mathbf{u}\|_{L^2} \sim m^k \|\mathbf{u}\|_{L^2}.
\end{equation}
Because the local resolution within each pore space is fixed, the global mesh size satisfies $h \propto m^{-1}$. By standard interpolation theory for $P_2$ elements, the absolute best-approximation errors for the velocity field in the $L^2$- and $H^1$-norms are bounded respectively by:
\begin{equation} \label{eq:vel_interp_abs}
    \begin{aligned}
        \|\mathbf{u} - \mathbf{u}_h\|_{L^2} & \le C h^3 |\mathbf{u}|_{H^3} \sim (m^{-1})^3 \left( m^3 \|\mathbf{u}\|_{L^2} \right) \sim \|\mathbf{u}\|_{L^2},                        \\
        |\mathbf{u} - \mathbf{u}_h|_{H^1}   & \le C h^2 |\mathbf{u}|_{H^3} \sim (m^{-1})^2 \left( m^3 \|\mathbf{u}\|_{L^2} \right) = m \|\mathbf{u}\|_{L^2} \sim |\mathbf{u}|_{H^1}.
    \end{aligned}
\end{equation}
Dividing these absolute errors by their corresponding norms mathematically demonstrates that the relative velocity errors are fundamentally invariant with respect to the obstacle density $m$:
\begin{equation} \label{eq:vel_rel_error}
    \frac{\|\mathbf{u} - \mathbf{u}_h\|_{L^2}}{\|\mathbf{u}\|_{L^2}} = \mathcal{O}(1), \quad \frac{|\mathbf{u} - \mathbf{u}_h|_{H^1}}{|\mathbf{u}|_{H^1}} = \mathcal{O}(1).
\end{equation}
This theoretical derivation is corroborated by Table \ref{tab:error_amplification}, where the relative velocity errors (Rel. $\mathbf{u}$) stabilize around $\sim 1.2\%$ for $L^2$ and $\sim 10\%$ for $H^1$ across all densities and both setups. Consequently, Table \ref{tab:error_amplification_orders} shows that their convergence orders stagnate at $0$.

\paragraph{Step 3: The Asymmetric inf-sup Pollution on Pressure}
Crucially, the pressure approximation is not independent; it is contaminated by the absolute velocity error through the residual duality pairing. Bounding this yields the standard {a priori} estimate:
\begin{equation} \label{eq:p_apriori}
    \| p - p_h \|_{L^2} \le C \underbrace{\beta_h^{-1}}_{\mathcal{O}(m)} \times \left( \|\nabla(\mathbf{u} - \mathbf{u}_h)\|_{L^2} + \inf_{q_h} \| p - q_h \|_{L^2} \right).
\end{equation}
In Setup B, enforcing a macroscopic parabolic profile triggers a severe boundary layer resonance. We observe that the absolute velocity error subsequently diverges: $\|\nabla(\mathbf{u} - \mathbf{u}_h)\|_{L^2} = \mathcal{O}(m)$ (seen in Table \ref{tab:error_amplification}, jumping consistently from $2.76$ at $m=4$ to an extreme $83.8$ at $m=128$). Amplified by the degraded inf-sup constant $\beta_h^{-1} = \mathcal{O}(m)$, this massive microscopic error causes the absolute pressure error to grow quadratically. Because the reference pressure inherently grows as $\mathcal{O}(m^2)$ (Darcy's law), the relative pressure error completely stagnates:
\begin{equation} \label{eq:p_stagnation}
    \frac{\| p - p_h \|_{L^2}}{\|p_{\mathrm{ref}}\|_{L^2}} \sim \frac{\mathcal{O}(m^2)}{\mathcal{O}(m^2)} = \mathcal{O}(1).
\end{equation}
This theoretical deadlock is corroborated by Table \ref{tab:error_amplification_orders}, where the pressure convergence order collapses entirely to $0.03$.

Surprisingly, in Setup A, the relative pressure error successfully evades this stagnation. We hypothesize that the natural Neumann condition induces a compatible macroscopic field that allows for a weak homogenization cancellation, effectively suppressing the velocity pollution term prior to the Cauchy-Schwarz bounding. However, establishing a rigorous two-scale convergence argument to fully capture this phenomenon falls beyond the primary scope of this paper, and we leave its formal mathematical resolution as an open question.

From a practical standpoint, the flow in realistic DLD devices is intrinsically driven by a fixed pressure drop. Setup B primarily serves as an idealized mathematical stress-test to demonstrate the extreme sensitivity of multiscale Stokes systems to macroscopic constraints. Therefore, we purposefully focus our discussion on the continuous stability properties and physical implications under the natural, pressure-driven regime.

Under pressure-driven (Neumann) conditions, the relative pressure error retains first-order convergence because the macroscopic field is compatible with the natural boundary formulation. Under velocity-driven (Dirichlet) conditions, the inf-sup penalty amplifies the absolute velocity error by $\mathcal{O}(m)$, matching the Darcy pressure growth and driving the relative pressure convergence rate to $0.03$. The penalty is a structural incompatibility between macroscopic Dirichlet constraints and multiscale Stokes systems, not a discretization defect.
\begin{table}[htbp]
    \centering
    \caption{Validation 2: Error amplification with fixed local resolution ($N_g=6$). As obstacle density $m$ increases, the absolute physical norms follow consistent Darcy scalings ($\|p\| \sim m^2$) in both setups. The relative velocity errors remain stationary ($\approx 1.2\% - 1.4\%$), demonstrating that the spatial discretization capacity is invariant. In contrast, the relative pressure errors suffer severe degradation under Dirichlet constraints due to the inf-sup penalty.}
    \label{tab:error_amplification}
    \renewcommand{\arraystretch}{1.2}
    \resizebox{\textwidth}{!}{
        \begin{tabular}{cc|cc|ccc|cc}
            \toprule
            \multicolumn{9}{c}{\textbf{Setup A: Pressure-Driven (Neumann, Pressure scaled by $m^2$)}}                                                                                                                                                             \\
            \midrule
            $m$ & $h_{\max}$ & $\|p_{\mathrm{ref}}\|_{L^2}$ & \textbf{Rel. $p$ ($L^2$)} & $\|\mathbf{u}_{\mathrm{ref}}\|_{L^2}$ & Abs. $\mathbf{u}$ ($L^2$) & \textbf{Rel. $\mathbf{u}$ ($L^2$)} & Abs. $\mathbf{u}$ ($H^1$) & \textbf{Rel. $\mathbf{u}$ ($H^1$)} \\
            \midrule
            2   & 0.1074     & 4.97e-01                     & \textbf{1.41e-02}         & 5.75e-03                              & 8.21e-05                  & \textbf{1.43e-02}                  & 8.82e-03                  & \textbf{1.07e-01}                  \\
            4   & 0.0549     & 1.91e+00                     & \textbf{7.85e-03}         & 7.08e-03                              & 8.64e-05                  & \textbf{1.22e-02}                  & 1.85e-02                  & \textbf{9.81e-02}                  \\
            8   & 0.0275     & 7.60e+00                     & \textbf{4.34e-03}         & 7.60e-03                              & 9.87e-05                  & \textbf{1.30e-02}                  & 4.09e-02                  & \textbf{1.03e-01}                  \\
            12  & 0.0184     & 1.72e+01                     & \textbf{2.97e-03}         & 7.79e-03                              & 1.02e-04                  & \textbf{1.31e-02}                  & 6.23e-02                  & \textbf{1.02e-01}                  \\
            16  & 0.0137     & 3.08e+01                     & \textbf{2.18e-03}         & 7.87e-03                              & 9.82e-05                  & \textbf{1.25e-02}                  & 8.10e-02                  & \textbf{9.91e-02}                  \\
            32  & 0.0069     & 1.24e+02                     & \textbf{1.13e-03}         & 8.01e-03                              & 1.06e-04                  & \textbf{1.32e-02}                  & 1.71e-01                  & \textbf{1.03e-01}                  \\
            64  & 0.0034     & 4.98e+02                     & \textbf{5.69e-04}         & 8.07e-03                              & 1.05e-04                  & \textbf{1.31e-02}                  & 3.44e-01                  & \textbf{1.03e-01}                  \\
            128 & 0.0017     & 2.00e+03                     & \textbf{2.98e-04}         & 8.11e-03                              & 1.05e-04                  & \textbf{1.30e-02}                  & 6.85e-01                  & \textbf{1.02e-01}                  \\
            \midrule
            \multicolumn{9}{c}{\textbf{Setup B: Velocity-Driven Pipe Flow (Fixed Macroscopic Profile $\mathbf{u} \sim 1$)}}                                                                                                                                       \\
            \midrule
            $m$ & $h_{\max}$ & $\|p_{\mathrm{ref}}\|_{L^2}$ & \textbf{Rel. $p$ ($L^2$)} & $\|\mathbf{u}_{\mathrm{ref}}\|_{L^2}$ & Abs. $\mathbf{u}$ ($L^2$) & \textbf{Rel. $\mathbf{u}$ ($L^2$)} & Abs. $\mathbf{u}$ ($H^1$) & \textbf{Rel. $\mathbf{u}$ ($H^1$)} \\
            \midrule
            2   & 0.1074     & 9.60e+01                     & \textbf{1.45e-02}         & 1.10e+00                              & 1.53e-02                  & \textbf{1.39e-02}                  & 1.69e+00                  & \textbf{1.07e-01}                  \\
            4   & 0.0549     & 2.84e+02                     & \textbf{8.48e-03}         & 1.05e+00                              & 1.25e-02                  & \textbf{1.19e-02}                  & 2.76e+00                  & \textbf{9.83e-02}                  \\
            8   & 0.0275     & 1.02e+03                     & \textbf{5.49e-03}         & 1.02e+00                              & 1.29e-02                  & \textbf{1.27e-02}                  & 5.48e+00                  & \textbf{1.03e-01}                  \\
            12  & 0.0184     & 2.23e+03                     & \textbf{4.41e-03}         & 1.01e+00                              & 1.29e-02                  & \textbf{1.28e-02}                  & 8.07e+00                  & \textbf{1.03e-01}                  \\
            16  & 0.0137     & 3.92e+03                     & \textbf{3.62e-03}         & 1.00e+00                              & 1.22e-02                  & \textbf{1.22e-02}                  & 1.03e+01                  & \textbf{9.92e-02}                  \\
            32  & 0.0069     & 1.54e+04                     & \textbf{3.48e-03}         & 9.96e-01                              & 1.29e-02                  & \textbf{1.29e-02}                  & 2.13e+01                  & \textbf{1.03e-01}                  \\
            64  & 0.0034     & 6.13e+04                     & \textbf{3.41e-03}         & 9.92e-01                              & 1.27e-02                  & \textbf{1.28e-02}                  & 4.23e+01                  & \textbf{1.03e-01}                  \\
            128 & 0.0017     & 2.44e+05                     & \textbf{3.34e-03}         & 9.91e-01                              & 1.26e-02                  & \textbf{1.27e-02}                  & 8.38e+01                  & \textbf{1.03e-01}                  \\
            \bottomrule
        \end{tabular}
    }
\end{table}

\begin{table}[htbp]
    \centering
    \caption{Empirical convergence rates for the relative errors calculated from Validation 2. As dictated by \eqref{eq:vel_interp_abs}, the relative velocity convergence strictly stagnates at zero ($\mathcal{O}(1)$) across all resolutions up to $m=128$. Consequently, under Setup A, the relative pressure error elegantly maintains optimal 1st-order convergence. Under Setup B, the extreme absolute velocity error, amplified by the inf-sup penalty, matches the physical Darcy pressure explosion, driving the pressure convergence rate to a complete standstill ($0.03$).}
    \label{tab:error_amplification_orders}
    \renewcommand{\arraystretch}{1.2}
    \begin{tabular}{cc | ccc | ccc}
        \toprule
        \multirow{2}{*}{$m$} & \multirow{2}{*}{$h_{\max}$} & \multicolumn{3}{c|}{\textbf{Setup A: Pressure-Driven}} & \multicolumn{3}{c}{\textbf{Setup B: Velocity-Driven}}                                                                                                      \\
        \cmidrule{3-8}
                             &                             & Order ($p$)                                            & Order ($\mathbf{u}_{L^2}$)                            & Order ($\mathbf{u}_{H^1}$) & Order ($p$) & Order ($\mathbf{u}_{L^2}$) & Order ($\mathbf{u}_{H^1}$) \\
        \midrule
        2                    & 0.1074                      & --                                                     & --                                                    & --                         & --          & --                         & --                         \\
        4                    & 0.0549                      & 0.87                                                   & \phantom{-}0.23                                       & \phantom{-}0.12            & 0.80        & \phantom{-}0.23            & \phantom{-}0.12            \\
        8                    & 0.0275                      & 0.85                                                   & -0.08                                                 & -0.06                      & 0.62        & -0.09                      & -0.06                      \\
        12                   & 0.0184                      & 0.94                                                   & -0.01                                                 & \phantom{-}0.00            & 0.55        & -0.02                      & \phantom{-}0.00            \\
        16                   & 0.0137                      & 1.06                                                   & \phantom{-}0.16                                       & \phantom{-}0.11            & 0.67        & \phantom{-}0.16            & \phantom{-}0.11            \\
        32                   & 0.0069                      & 0.94                                                   & -0.08                                                 & -0.05                      & 0.05        & -0.08                      & -0.06                      \\
        64                   & 0.0034                      & 0.99                                                   & \phantom{-}0.01                                       & \phantom{-}0.00            & 0.03        & \phantom{-}0.01            & \phantom{-}0.00            \\
        128                  & 0.0017                      & 0.94                                                   & \phantom{-}0.01                                       & \phantom{-}0.01            & 0.03        & \phantom{-}0.01            & \phantom{-}0.01            \\
        \bottomrule
    \end{tabular}
\end{table}
\subsection{Buffer Zones Do Not Cure the Pressure Stagnation}
\label{subsec:buffer_configurations}

In practical multiscale fluid simulations, appending upstream and downstream buffer zones is a standard computational practice to mitigate boundary-induced numerical artifacts. Depending on the discipline, two distinct buffer configurations are commonly employed:
\begin{enumerate}
    \item \textbf{Configuration 1: Patterned Buffer (Classical Oversampling).} Widely used in homogenization and multiscale finite element methods \cite{Hou1997}, this involves extending the domain using the identical periodic micro-pillar array. The premise is that high-frequency boundary layers decay exponentially within this patterned buffer.
    \item \textbf{Configuration 2: Un-patterned Buffer (Realistic Empty Channel).} Standard in engineering CFD, this involves appending obstacle-free straight channels. The intuitive expectation is that the flow will naturally transition from a fully developed Poiseuille profile into the complex array.
\end{enumerate}
As illustrated in Fig.~\ref{fig:buffer_configurations}, the computational domain is extended by a buffer zone of length $L_b = 0.5$ at both ends.
\begin{figure}[htbp]
    \centering
    \begin{subfigure}[b]{0.8\textwidth}
        \centering
        \includegraphics[width=\textwidth]{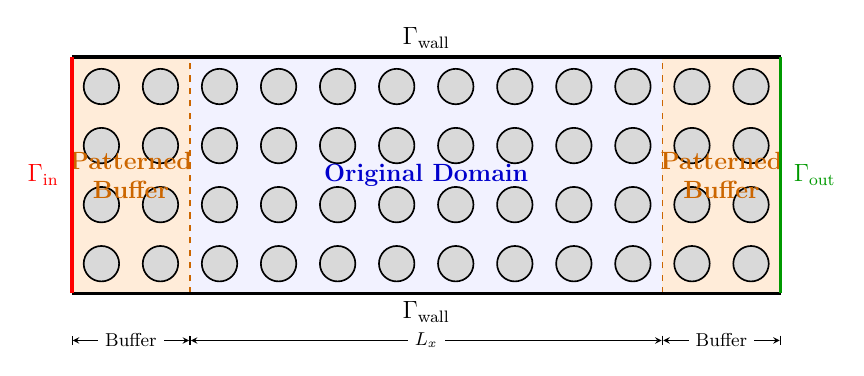}
        \caption{Configuration 1: Patterned Buffer setup. The central domain $(0, L_x) \times (0, L_y)$ is extended by periodic buffer zones of width $L_b = 0.5$ at both ends. For $\epsilon = 1/m$, this results in a total simulation domain of $3m \times m$ unit cells.}
        \label{fig:patterned_buffer}
    \end{subfigure}

    \vspace{0.5cm} 

    \begin{subfigure}[b]{0.8\textwidth}
        \centering
        \includegraphics[width=\textwidth]{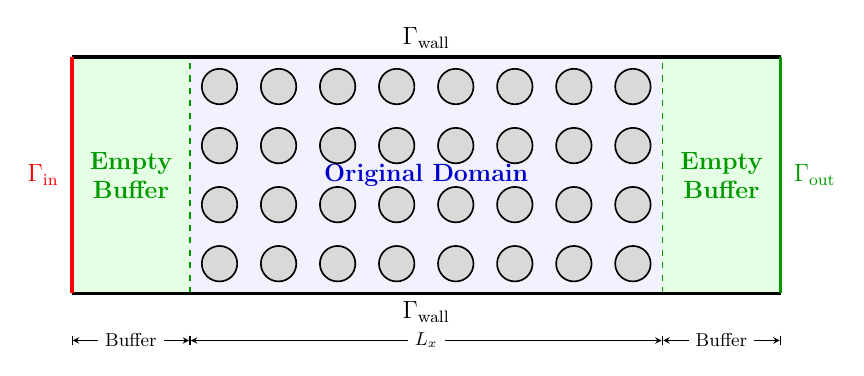}
        \caption{Configuration 2: Un-patterned Buffer setup. The central domain containing $2m \times m$ unit cells is augmented with obstacle-free straight channels of width $L_b = 0.5$ at the inlet and outlet to allow for flow development.}
        \label{fig:empty_buffer}
    \end{subfigure}

    \caption{Schematic of the extended computational domains with different buffer configurations. The core geometry remains a DLD micro-pillar array with unit cell size $\epsilon = 1/m$.}
    \label{fig:buffer_configurations}
\end{figure}

To evaluate whether these conventional remedies can mitigate the error degradation observed under macroscopic Dirichlet constraints, we augmented Setup B by appending buffer regions of length $L_b = 0.5$ to both ends of the domain. We extract and strictly evaluate the relative errors within the central interior array ($\Omega_{\mathrm{inner}}$).

\begin{table}[htbp]
    \centering
    \caption{Error amplification and empirical convergence rates under two distinct extended buffer configurations ($L_b = 0.5$). Regardless of whether the buffer is patterned (oversampling) or un-patterned (empty channel), the relative pressure error in the interior domain severely stagnates. Both convergence rates collapse toward zero as the density $m$ increases.}
    \label{tab:extended_channel_comparison}
    \renewcommand{\arraystretch}{1.2}
    \resizebox{\textwidth}{!}{
        \begin{tabular}{cc | cc | cc | cc | cc}
            \toprule
            \multicolumn{2}{c|}{} & \multicolumn{4}{c|}{\textbf{Config 1: Patterned Buffer (Oversampling)}} & \multicolumn{4}{c}{\textbf{Config 2: Un-patterned Buffer (Empty Channel)}}                                                                                                                                                                                                 \\
            \cmidrule(lr){3-6} \cmidrule(lr){7-10}
            $m$                   & $h_{\max}$                                                              & \textbf{Rel. $p$ ($L^2$)}                                                  & Order ($p$)   & \textbf{Rel. $\mathbf{u}$ ($H^1$)} & Order ($\mathbf{u}_{H^1}$) & \textbf{Rel. $p$ ($L^2$)} & Order ($p$)   & \textbf{Rel. $\mathbf{u}$ ($H^1$)} & Order ($\mathbf{u}_{H^1}$) \\
            \midrule
            2                     & 0.107 -- 0.110                                                          & 1.47e-02                                                                   & --            & 1.07e-01                           & --                         & 1.58e-02                  & --            & 1.07e-01                           & --                         \\
            4                     & 0.055                                                                   & 8.56e-03                                                                   & 0.80          & 9.93e-02                           & \phantom{-}0.11            & 9.24e-03                  & 0.78          & 1.06e-01                           & \phantom{-}0.01            \\
            8                     & 0.027                                                                   & 5.47e-03                                                                   & 0.65          & 1.02e-01                           & -0.04                      & 5.52e-03                  & 0.74          & 1.04e-01                           & \phantom{-}0.03            \\
            12                    & 0.018                                                                   & 4.39e-03                                                                   & 0.54          & 1.03e-01                           & -0.02                      & 4.45e-03                  & 0.54          & 1.03e-01                           & \phantom{-}0.02            \\
            16                    & 0.014                                                                   & 3.63e-03                                                                   & 0.66          & 9.98e-02                           & \phantom{-}0.11            & 3.82e-03                  & 0.66          & 1.02e-01                           & \phantom{-}0.04            \\
            32                    & 0.007                                                                   & 3.47e-03                                                                   & \textbf{0.06} & 1.04e-01                           & -0.06                      & 3.45e-03                  & \textbf{0.15} & 1.03e-01                           & -0.01                      \\
            \bottomrule
        \end{tabular}
    }
\end{table}

As documented in Table \ref{tab:extended_channel_comparison}, while both buffering strategies successfully stabilize the velocity approximation, the relative pressure error exhibits a persistent zero-order stagnation, refusing to converge beyond $\sim 3.4 \times 10^{-3}$.

These empirical findings serve as a necessary cautionary note for the simulation of highly heterogeneous systems. The fact that the pressure error remains locked at $\mathcal{O}(1)$ despite sufficient physical buffering presents a significant analytical challenge. The precise mathematical mechanisms driving this persistent boundary-induced stagnation, as well as the design of universally compatible boundary treatments for multiscale Stokes formulations, remain unclear within our current theoretical framework and are left as open problems for future investigation.

Neither patterned nor un-patterned buffer zones cure the pressure stagnation under macroscopic Dirichlet constraints. The error floor at $\sim 3.4 \times 10^{-3}$ is a consequence of the inf-sup penalty, not a boundary artifact that physical buffering can remove.

\subsection{Algorithmic Implications: Adaptive AL Scaling}\label{subsec:al_scaling}

Consider the discrete saddle-point system arising from a stable mixed finite element discretization (e.g., Taylor-Hood $P_2-P_1$ elements) of the Stokes problem on the pillar-array domain $\Omega_m$. The conditioning of the standard discrete pressure Schur complement $S = B A^{-1} B^T$ is strictly bounded by the discrete inf-sup constant $\beta_h$, yielding $\kappa(M_p^{-1} S) \propto \beta_h^{-2}$. According to our main theorems, since the continuous constant degrades as $\beta = \Theta(m^{-1})$, any conforming discrete space inherently suffers from $\beta_h \le \mathcal{O}(m^{-1})$, leading to $\kappa(M_p^{-1}S) \ge \mathcal{O}(m^2)$. This severe ill-conditioning causes the stagnation of standard Krylov subspace methods.

To restore the robustness of the outer iterations, the Augmented Lagrangian (AL) method introduces a grad-div stabilization term $\gamma B^T \operatorname{diag}(M_p)^{-1} B \mathbf{u}_h$ into the momentum equation. The modified Schur complement $S_{\gamma}$ exhibits a significantly improved condition number bound \cite{Benzi2006}:
\begin{equation}\label{eq:al_condition_bound}
    \kappa(M_p^{-1}S_{\gamma}) \le C \left( 1 + \frac{1}{\gamma \mu^{-1} \beta_h^2} \right),
\end{equation}
where $C$ is a constant independent of the mesh size $h$ and the geometry.

In standard literature, the penalty parameter $\gamma$ is often chosen heuristically. However, leveraging our pre-asymptotic bound $\beta_h \propto m^{-1}$, we can now rigorously derive the optimal strategy. To ensure that the outer Krylov solver remains completely robust and independent of the pillar density $m$ (i.e., enforcing $\kappa(M_p^{-1}S_{\gamma}) = \mathcal{O}(1)$), the penalty parameter must be adaptively scaled according to:
\begin{equation}\label{eq:gamma_scaling}
    \gamma = \gamma_0 \cdot m^2 \propto \epsilon^{-2},
\end{equation}
where $\gamma_0 > 0$ is a baseline constant of $\mathcal{O}(1)$. This strategy directly translates our continuous PDE analysis into an actionable, parameter-free stabilization strategy for microfluidic simulations.

\subsection{AL Scaling Achieves Density-Independent Convergence on Square Arrays}\label{subsec:validation_al_square}

Before extending our methodology to complex microfluidic designs, we first validate the proposed parameter-free Augmented Lagrangian (AL) preconditioner on a standard benchmark: a regular square micro-pillar array ($\delta = 0$). We maintain a constant solid packing ratio $\kappa = 0.25$ (i.e., the pillar radius is fixed at $0.25$ times the characteristic unit cell diameter) and systematically refine the characteristic mesh density $m$ from $2$ to $64$.

To rigorously evaluate the algorithmic scalability and preclude any false convergence artifacts (such as volumetric locking inherent to high penalty parameters in $P_2-P_1$ elements), the FGMRES solver is driven by the unpreconditioned physical residual, enforced with a stringent absolute tolerance of $10^{-10}$. Moreover, to isolate the macroscopic conditioning of the pressure Schur complement from the specific choice of inner algebraic preconditioners, we employ a robust sparse direct solver (e.g., UMFPACK) for all inner velocity block inversions. The development of optimal, scalable multigrid preconditioners for the highly ill-conditioned AL velocity block in these densely packed geometries is a distinct and substantial challenge \cite{Benzi2006, Farrell2019}, which falls outside the scope of this theoretical benchmark and is deferred to our future algorithmic research.

\begin{figure}[htbp]
    \centering
    \begin{subfigure}[b]{0.48\textwidth}
        \centering
        \includegraphics[width=\textwidth]{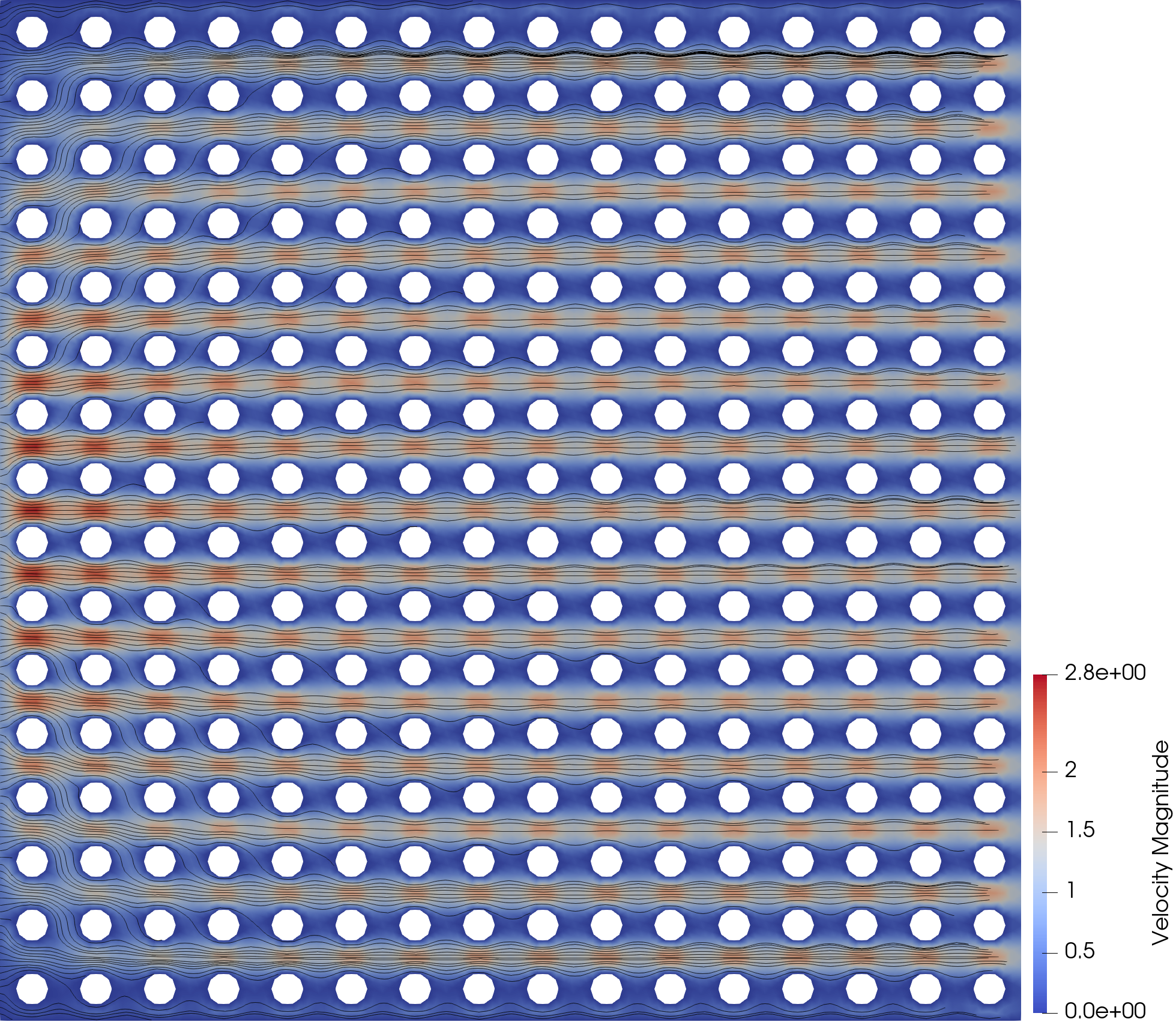}
        \caption{$m = 16$ (Total DoFs: 109,983)}
        \label{fig:square_m16}
    \end{subfigure}
    \hfill 
    \begin{subfigure}[b]{0.48\textwidth}
        \centering
        \includegraphics[width=\textwidth]{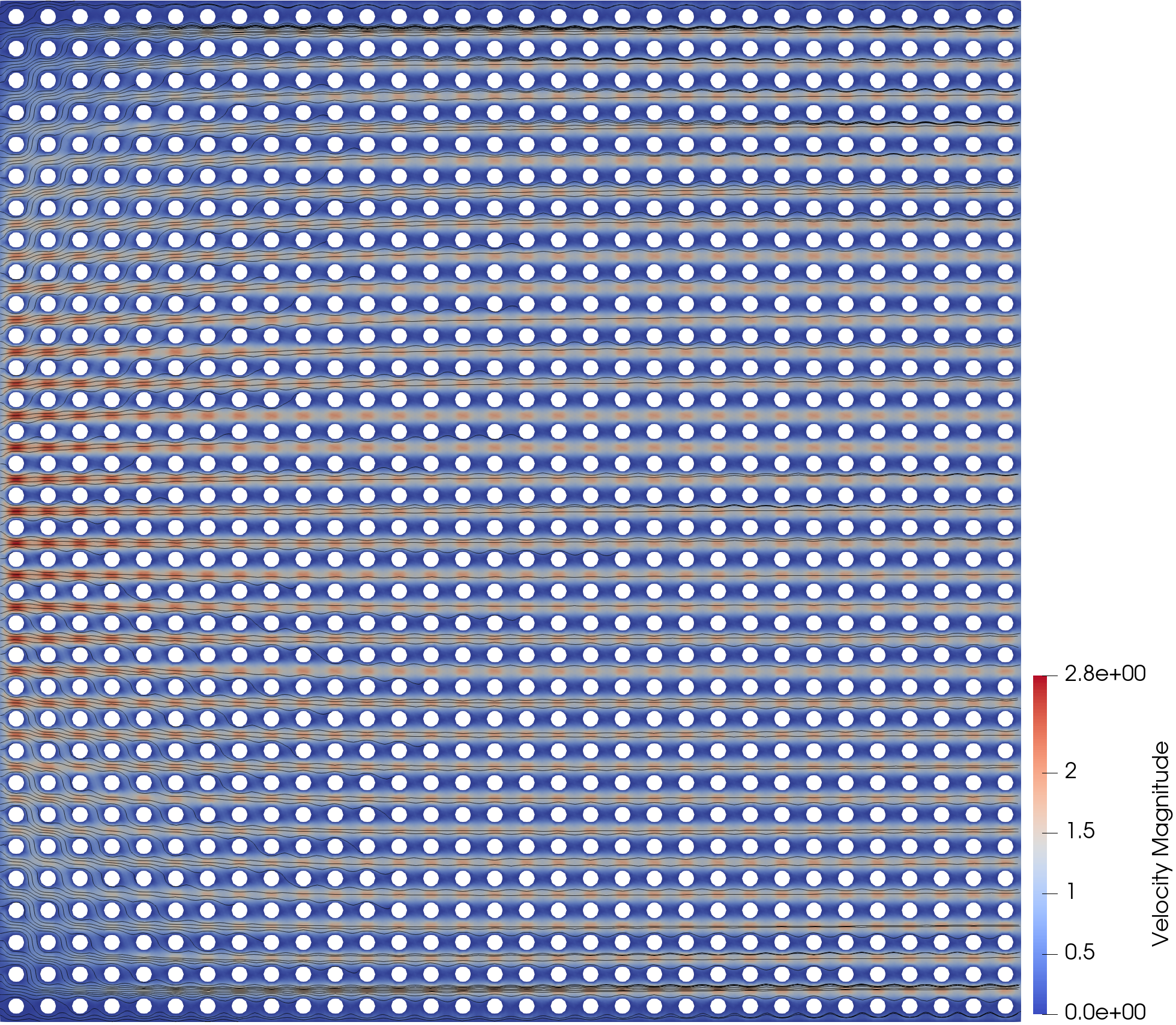}
        \caption{$m = 32$ (Total DoFs: 441,508)}
        \label{fig:square_m32}
    \end{subfigure}

    \caption{Velocity magnitude contours and streamlines of the steady Stokes flow through the square array ($\delta = 0, \kappa = 0.25$). The flow fields computed on the intermediate-density configuration (\subref{fig:square_m16}) and the high-density configuration (\subref{fig:square_m32}) share an identical colormap range. The results demonstrate that the highly scaled AL preconditioner accurately resolves the complex fluid tortuosity and local shear layers without exhibiting any numerical locking, even as the system approaches half a million degrees of freedom.}
    \label{fig:square_flow_comparison}
\end{figure}

\begin{table}[htbp]
    \centering
    \caption{Iteration count comparison for the standard square micro-pillar array ($\delta = 0, \kappa = 0.25$). The Scaled AL method effectively mitigates the ill-conditioning, achieving optimal algorithmic scalability as the degrees of freedom (DoFs) scale up to 1.85 million.}
    \label{tab:solver_iterations_square}
    \begin{tabular}{cccc}
        \toprule
        Density ($m \times m$) & Total DoFs & Std. Stokes Iters ($\gamma=0$) & AL-Stokes Iters ($\gamma = m^2$) \\
        \midrule
        $2 \times 2$           & 1,795      & 34                             & 25                               \\
        $4 \times 4$           & 7,183      & 53                             & 25                               \\
        $8 \times 8$           & 29,071     & 86                             & 25                               \\
        $16 \times 16$         & 116,872    & 154                            & 23                               \\
        $32 \times 32$         & 468,574    & 277                            & 22                               \\
        $64 \times 64$         & 1,852,843  & 437                            & 22                               \\
        \bottomrule
    \end{tabular}
\end{table}

As shown in Table \ref{tab:solver_iterations_square}, the standard formulation ($\gamma=0$) suffers algorithmic degradation, with iteration counts growing from 34 to 437 as the obstacle density increases. This deterioration is a direct algebraic consequence of the degrading inf-sup stability. The Scaled AL method neutralizes this saddle-point pathology. Dictated by our strategy ($\gamma = m^2$), the AL preconditioner remains bounded and even slightly decreases, stabilizing at 22 steps for a system approaching 1.85 million DoFs. This density-independent convergence establishes a baseline for simulating more complex microfluidic topologies.

The $\gamma = m^2$ scaling reduces FGMRES iterations from 437 to 22 at 1.85M DoFs and eliminates the $\mathcal{O}(m)$ growth observed in the standard formulation. Iteration counts are bounded independent of pillar density, confirming the algebraic implication of our sharp inf-sup bound.

\subsection{AL Robustness Holds Under Geometric Asymmetry ($\delta = 1/3$)}\label{subsec:validation_al_dld}

To definitively demonstrate the generalizability and practical engineering efficacy of our parameter-free stabilization strategy, we turn to a realistic Deterministic Lateral Displacement (DLD) microfluidic array. Unlike standard Cartesian grids, DLD arrays are characterized by a shifted Bravais lattice ($\delta > 0$), which breaks geometric symmetry and introduces complex, tortuous streamlines responsible for particle fractionation.

We construct a sequence of DLD pillar arrays with a characteristic row shift fraction $\delta = 1/3$ and maintain the solid geometric ratio $\kappa = 0.25$. The mesh density $m$ is scaled from $2$ up to an extreme case of $64$, culminating in a highly ill-conditioned saddle-point system with over 1.76 million degrees of freedom. Crucially, to preclude any spurious convergence artifacts (e.g., false volumetric locking) that can be induced by the highly scaled penalty parameter ($\gamma = m^2$) in the $P_2-P_1$ space, we strictly evaluate the unpreconditioned physical residual. The FGMRES solver is enforced to satisfy a stringent absolute tolerance of $10^{-10}$ alongside a relative tolerance of $10^{-8}$.
\begin{figure}[htbp]
    \centering
    \begin{subfigure}[b]{0.48\textwidth}
        \centering
        \includegraphics[width=\textwidth]{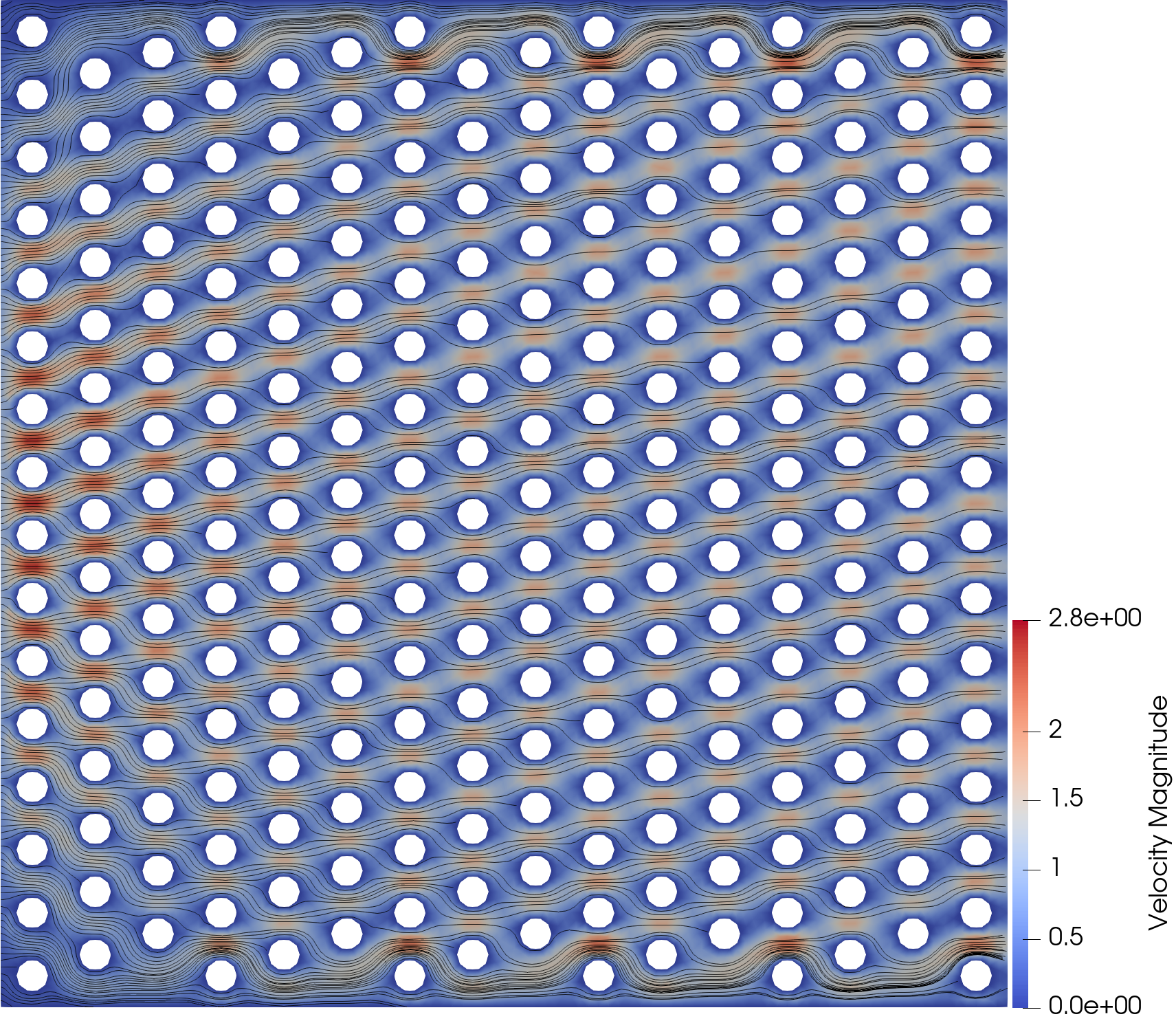}
        \caption{$m = 16$ (Total DoFs: 109,983)}
        \label{fig:dld_m16}
    \end{subfigure}
    \hfill
    \begin{subfigure}[b]{0.48\textwidth}
        \centering
        \includegraphics[width=\textwidth]{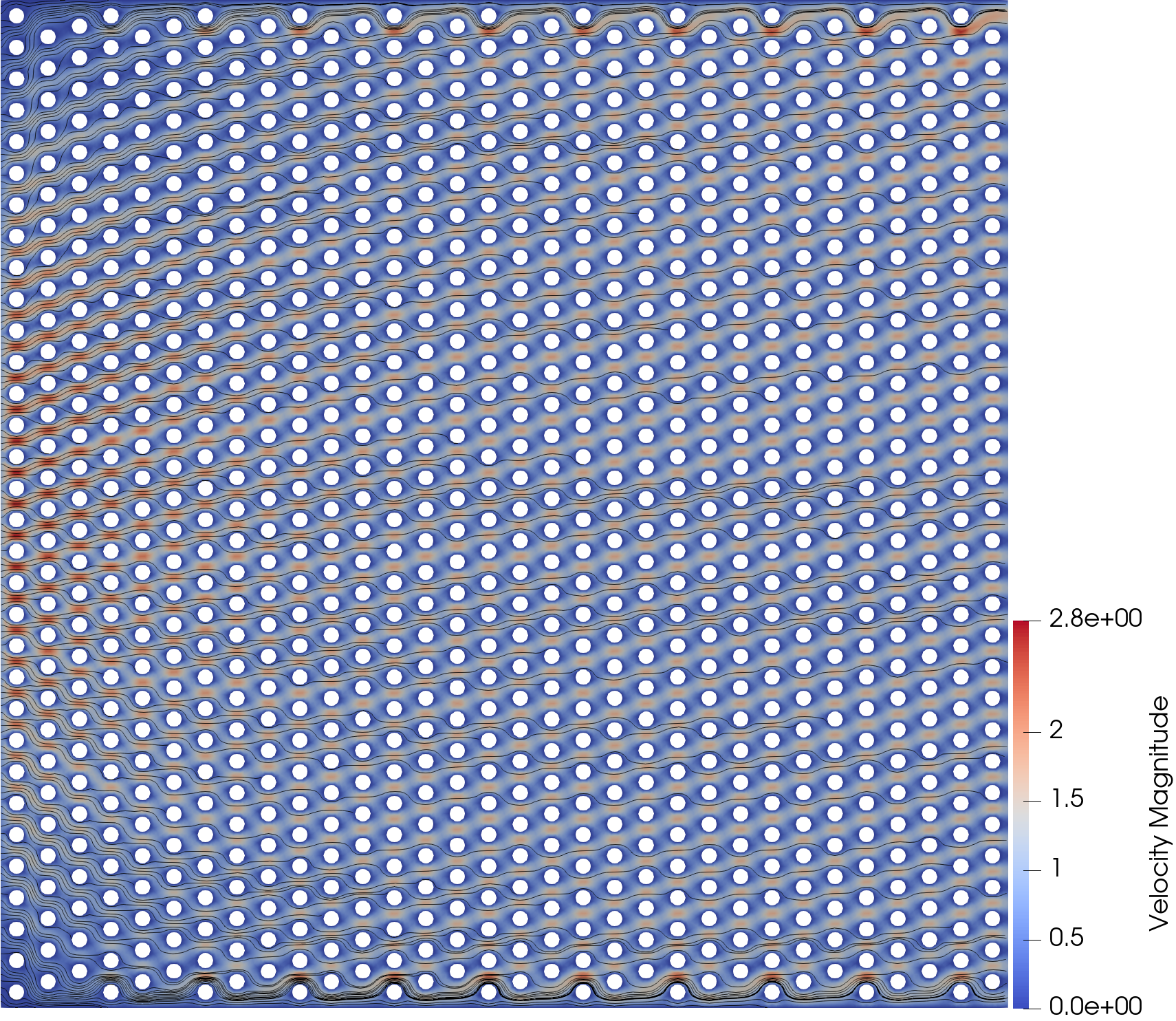}
        \caption{$m = 32$ (Total DoFs: 441,508)}
        \label{fig:dld_m32}
    \end{subfigure}
    \caption{Velocity magnitude contours and streamlines of the steady Stokes flow through the DLD array ($\delta = 1/3, \kappa = 0.25$). The flow fields computed on the intermediate-density configuration (\subref{fig:dld_m16}) and the high-density configuration (\subref{fig:dld_m32}) share an identical colormap range. The results demonstrate that the highly scaled AL preconditioner accurately resolves the complex fluid tortuosity and local shear layers without exhibiting any numerical locking, even as the system approaches half a million degrees of freedom.}
    \label{fig:dld_flow_comparison}
\end{figure}

\begin{table}[htbp]
    \centering
    \caption{Comparison of FGMRES outer iterations for a DLD array with a row shift fraction $\delta = 1/3$. Even with broken symmetry, highly tortuous flow paths, and stringent convergence criteria ($10^{-10}$ absolute tolerance), the AL method exhibits robust algorithmic scalability.}
    \label{tab:solver_iterations_dld}
    \begin{tabular}{cccc} 
        \toprule
        Density ($m \times m$) & Total DoFs & Std. Stokes Iters ($\gamma=0$) & AL-Stokes Iters ($\gamma = m^2$) \\
        \midrule
        $2 \times 2$           & 1,755      & 33                             & 24                               \\
        $4 \times 4$           & 6,860      & 57                             & 28                               \\
        $8 \times 8$           & 27,467     & 95                             & 27                               \\
        $16 \times 16$         & 109,983    & 175                            & 27                               \\
        $32 \times 32$         & 441,508    & 324                            & 26                               \\
        $64 \times 64$         & 1,766,167  & 687                            & 24                               \\
        \bottomrule
    \end{tabular}
\end{table}

Table \ref{tab:solver_iterations_dld} shows that the geometric asymmetry and tortuous flow paths of the DLD device exacerbate the algebraic failure of the standard decoupled formulation. Compared to the square array, the standard Stokes iterations degrade more aggressively, from 33 to 687 iterations as pillar density increases. The Augmented Lagrangian method, governed by $\gamma = m^2$, remains bounded under this added geometric complexity, stabilizing at 24 iterations for the 1.76M-DoF system. This experiment corroborates our framework: a parameter-free stabilization strategy insensitive to both the obstacle density and the geometric configurations of modern microfluidic simulations.

The AL scaling is insensitive to geometric asymmetry. Iteration counts remain bounded at 24 on a 1.77M-DoF DLD array, while the standard formulation degrades to 687. The $\gamma \propto m^2$ rule addresses the geometric ill-conditioning itself, not a symmetry-specific property of the square lattice.

\section{Conclusions and Future Directions}\label{sec:conclusion}

In this paper, we have rigorously established a sharp pre-asymptotic estimate for the continuous inf-sup constant in densely packed microfluidic arrays. We proved that the constant deteriorates at a sharp rate of $\Theta(m^{-1})$ as the pillar density $m$ increases, exposing the fundamental geometric penalty that dictates the stability of the multiscale Stokes system.

Our numerical investigations map this stability degradation to its two concrete manifestations. On the discretization side, the sharp bound enters Brezzi's a priori estimate through the factor $\beta_h^{-1} = \mathcal{O}(m)$, and we observe an asymmetric pollution under macroscopic boundary conditions: Neumann pressure-drop inflows retain first-order pressure convergence, while Dirichlet inflow profiles drive the relative pressure error to a complete $\mathcal{O}(1)$ stagnation that buffer zones do not cure. On the algebraic side, the Schur complement condition scales as $\mathcal{O}(m^2)$, dictating the stagnation of standard iterative methods. Leveraging the sharp bound, we proposed an adaptively scaled Augmented Lagrangian (AL) method with $\gamma \propto m^2$. Experiments on square and asymmetric DLD arrays show that this parameter-free scaling neutralizes the geometric ill-conditioning and achieves density-independent convergence for systems with millions of degrees of freedom.

Future research will focus on developing geometry-aware multigrid preconditioners for the inner AL velocity blocks and extending this framework to multi-scale Fluid-Structure Interaction (FSI) problems. These advancements will enable predictive, fully resolved simulations of complex cell deformations and trajectories within intricate micro-channel topologies.

\appendix

\bibliographystyle{elsarticle-num}
\bibliography{reference}

\end{document}